\documentclass[dvipsnames, hidelinks]{article}
\usepackage{arxiv}  
\usepackage[utf8]{inputenc} 
\usepackage[T1]{fontenc}    
\usepackage{hyperref}       
\usepackage{url}            
\usepackage{booktabs}       
\usepackage{nicefrac}       
\usepackage{amsfonts,amssymb,mathtools,amsthm}  
\usepackage{marginnote}
\usepackage{url}
\usepackage{enumerate}  
\usepackage[
backend=biber,
style=alphabetic,
citestyle=alphabetic,
maxbibnames=99
]{biblatex}
\usepackage{tikz}   
\usepackage{xcolor} 
\usepackage[caption=false]{subfig}
\usepackage{circuitikz} 
\usepackage{bm} 
\usepackage{orcidlink}  
\RequirePackage{algorithm}[0.1]
\usepackage[noend]{algorithmic}  
\usepackage{dsfont} 
\allowdisplaybreaks 
\usepackage{multirow}   
\usepackage{cleveref}

\definecolor{mycol1}{HTML}{B0413E}  
\definecolor{mycol2}{HTML}{8ED081}  
\definecolor{mycol3}{HTML}{2D93AD}  
\definecolor{mycol4}{HTML}{A0A0A0}  

\theoremstyle{plain}
\newtheorem{theorem}{Theorem}

\newtheorem{lemma}{Lemma}

\theoremstyle{definition}
\newtheorem{definition}{Definition}
\newtheorem{observation}{Observation}
\newtheorem{example}{Example}
\theoremstyle{remark}

\newcommand\freefootnote[1]{%
  \begin{NoHyper}
  \renewcommand\thefootnote{}\footnote{#1}%
  \addtocounter{footnote}{-1}%
  \end{NoHyper}
}

\title{Lower bounds for the integrality gap of the bi--directed cut formulation of the Steiner Tree Problem}

\author{
Ambrogio Maria Bernardelli$^{*}$ \orcidlink{0000-0002-2328-7062} \\
\texttt{ambrogiomaria.bernardelli@unipv.it} \And
Eleonora Vercesi$^{\dagger \ddagger}$ \\
\texttt{eleonora.vercesi@usi.ch}\And 
Stefano Gualandi$^{*} \orcidlink{0000-0003-1981-6782}$ \\
\texttt{stefano.gualandi@unipv.it}  \And 
Monaldo Mastrolilli$^{\S \ddagger}$\\ 
\texttt{monaldo@supsi.ch} \And 
Luca Maria Gambardella$^{\dagger\ddagger}$ \\
\texttt{luca.gambardella@usi.ch}}
\date{\today}

\begin{document}
\maketitle
\begin{abstract}

\noindent In this work, we study the metric Steiner Tree problem on graphs focusing on computing lower bounds for the integrality gap of the bi-directed cut (BCR) formulation and introducing a novel formulation, the Complete Metric (CM) model, specifically designed to address the weakness of the BCR formulation on metric instances.
A key contribution of our work is extending the Gap problem, previously explored in the context of the Traveling Salesman problems, to the metric Steiner Tree problem.
To tackle the Gap problem for Steiner Tree instances, we first establish several structural properties of the CM formulation. We then classify the isomorphism classes of the vertices within the CM polytope, revealing a correspondence between the vertices of the BCR and CM polytopes.
Computationally, we exploit these structural properties to design two complementary heuristics for finding nontrivial small metric Steiner instances with a large integrality gap.
We present several vertices for graphs with a number of nodes $\leq 10$, which realize the best-known lower bounds on the integrality gap for the CM and the BCR formulations.
We conclude the paper by presenting two new conjectures on the integrality gap of the BCR and CM formulations for small graphs.
\end{abstract}

\freefootnote{$^{*}$ Department of Mathematics ``Felice Casorati'', University of Pavia, 27100 Pavia, Italy}
\freefootnote{$^{\dagger}$ Faculty of Informatics, Università della Svizzera italiana, 6900 Lugano, Switzerland}
\freefootnote{$^{\S}$ Dipartimento Tecnologie innovative, Scuola universitaria professionale della Svizzera italiana, 6900 Lugano, Switzerland}
\freefootnote{$^{\ddagger}$ Istituto Dalle Molle di studi sull'intelligenza artificiale (IDSIA USI-SUPSI), 6900 Lugano, Switzerland}

\section{Introduction}
Given a graph $G=(V, E)$ with $n := |V| = |\{1, \dots, n\}|$ nodes and cost $c_{ij}\geq 0$ on each edge $\{i,j\} \in E$, and a subset of nodes $T \subset V$ with $t \coloneqq |T| \geq 2$, the Steiner Tree Problem (STP) consists of finding the minimum-cost tree that spans $T$.
The nodes in $T$ are called \emph{terminals}, and those in $V \setminus T$ are called \emph{Steiner nodes}.
Note that, by a relabelling of the vertices, we might just assume that $T = [t] := \{1, \dots, t\}$. In the following, we will often denote simply by $G$ an instance of the STP, that is, the graph, the edge costs, and the set of terminals.
The STP is NP-Hard, and the corresponding decision problem is NP-Complete~\cite{karp2010reducibility}.
The best-known polynomial-time algorithm for the STP guarantees an approximation ratio of 1.39~\cite{byrka2013steiner}, and improving this ratio is still an open problem.
STP is also MaxSNP-Hard~\cite{bern1989steiner}, implying that no approximation algorithm can exist unless $\text{P} = \text{NP}$. More specifically, it has been proved that there is no 1.01-factor approximation algorithm for the STP~\cite{chlebik2008steiner}. Finally, note that STP is fixed-parameter tractable with respect to $\vert T \vert$(See, e.g.,~\cite{dreyfus1971steiner}). The two special cases $\vert T \vert = 2$, which is the shortest path problem, and $\vert T \vert = n$, which is the minimum spanning tree, have dedicated polynomial-time algorithms developed through the years.

STP is often solved to optimality via integer linear programming with several formulations that have been proposed. For a catalog of them, we refer to~\cite{goemans1993catalog}, while for more recent surveys, to~\cite{Koch1998,ljubic2021solving}.

One of the most promising formulations is the Bidirected cut formulation (BCR), introduced by Edmonds~\cite{edmonds1967optimum}.
The core of the BCR formulation consists of fixing a root node $r \in T$, replacing each undirected edge $\{i, j\}$ with two oriented arcs $(i, j)$ and $(j, i)$, and introducing 0--1 flow variables $x_{ij}$ for each arc. 
If we denote by $A$ the set of oriented arcs, and by $\delta^{-}(W) := \{(i, j) \in A \; \vert \; i \not \in W, \; j \in W \}$ the set of arcs entering $W \subset V$, the BCR polytope is as follows:
	\begin{subequations}\label{eq:dcut-polytope}
		\begin{align}
			P_{BCR}(n,t) \coloneqq \{& x \in \mathbb{R}_{\geq 0}^m \, : \\    
			& x_{i j}+x_{j i} \leq 1, & \{i, j\} \in E, \label{eq:dcut_degree} \\
			& x\left(\delta^{-}(W)\right) \geq 1, & W \subset V \setminus \{r\}, \;W \cap T \neq \emptyset \} \label{eq:dcut_subt}
		\end{align}
	\end{subequations}
	where $m=|A|$, and we define the set of integer points of $P_{BCR}(n,t)$ as
	\begin{equation}
		S_{BCR}(n,t) \coloneqq P_{BCR}(n,t) \cap \{0,1\}^m.
	\end{equation}
	Given a STP instance $G$ with $n$ nodes and $t$ terminals and with $c$ as edge costs, we will denote with STP$(G)$ the cost of the subgraph of $G$ having the following properties: $(i)$ it is a tree spanning $T$ and $(ii)$ it is of minimum cost. Equivalently,
\begin{equation}\label{eq:opt-int}
	\text{STP}(G) \coloneqq \min \{c^\intercal x \, | \, x \in S_{BCR}(n,t)\}.
\end{equation}
Let also
\begin{equation}
	\text{opt}_{BCR}(G) \coloneqq \min \{c^\intercal x \, | \, x \in P_{BCR}(n,t)\}
\end{equation}be the optimal value of the linear relaxation of BCR. We will use the same notation for all of the formulations presented in this work. Note that we will drop the arguments $(n,t)$ when unnecessary or clear from the context, for better readability. 
Despite their exponential number, the cut constraints \eqref{eq:dcut_subt} can be separated in polynomial time using any max flow or min-cut algorithm, as detailed in~\cite{Koch1998}.
Note that it has been shown that the optimal value of BCR is independent of the choice of the root~\cite{goemans1993catalog}.
The state-of-the-art (parallel) implementation of a more tight version of the BCR model can be found in SCIP-Jack~\cite{gamrath2017scip}.
Other more recent approaches are reviewed in~\cite{ljubic2021solving}.

Whenever the graph $G$ is complete, and the edge costs are metric, we have a complete metric Steiner Tree problem.
The edge costs define a metric if they satisfy the following properties: (i) $c_{ij} = 0$ if and only if $i = j$; (ii) $c_{ij} \geq 0$ ({\it positivity}), (iii) $c_{ij} = c_{ji}$ ({\it symmetry}), (iv) $c_{ij} \leq c_{ik} + c_{jk}$ ({\it triangle inequality}).
Metric STP instances are relevant in particular for VLSI circuit design~\cite{Peyer2007}, and efficient combinatorial algorithms exist for rectilinear costs~\cite{hougardy2017dijkstra} and for packing Steiner trees~\cite{Grotschel1997}.
A review of results for metric (and rectilinear) STP is contained in~\cite{hougardy2017dijkstra}. 

In approximation algorithms, we are interested in studying the \emph{integrality gap} of an integer formulation, defined as the supremum among all the ratios between the optimal value of an integer solution and the optimal value of its natural LP relaxation.
We now define the integrality gap of the BCR formulation as the following
\begin{equation}\label{GAP}
\alpha := \sup_{G} \dfrac{\text{STP}(G)}{\text{opt}_{BCR}(G)}.
\end{equation}
For the BCR, the exact value of $\alpha$ is unknown, but it is 
bounded below by $\frac{6}{5}$~\cite{vicari2020simplexbasedsteinertree}, which improved the previous bound of  $\frac{36}{31} \cong 1.161$ in~\cite{byrka2013steiner}. Until 2024, the best--known upper bound for the integrality gap was 2. Recently,~\cite{byrka2024bidirectedcutrelaxationsteiner} demonstrated that this upper bound can be improved to 1.9988.

Note that proving that $\alpha < 1.39$ would lead to a better approximation algorithm with respect to the state of the art.
The lower bound introduced by~\cite{byrka2013steiner} is based on a recursive family of instances, depending on a parameter $p$ and having an integrality gap which tends asymptotically to $\frac{36}{31}$ for $n \rightarrow \infty$. 

\paragraph{Main Contributions}
This paper presents a novel formulation for the complete metric STP, called the Complete Metric (CM) formulation.
This formulation exploits the metric costs to define a polytope, denoted by $P_{CM}$ and defined in \eqref{eq:pcm}, having a smaller number of vertices compared to the polytope $P_{BCR}$ implied by the BCR formulation.
The main motivation for introducing our formulation is to enable a study on the integrality gap of small-size instances 
of the Steiner Tree problem by adapting the approach designed for the Symmetric TSP and presented in~\cite{boyd2005computing,boyd2007structure}.
Without our new CM formulation, it would be nearly impossible to use the method of~\cite{boyd2007structure} 
due to the number of vertices of $P_{BCR}$, which includes a huge number of feasible vertices for the cut constraints but which will never be optimal for any \emph{metric} cost vector.
For instance, \Cref{tab:vertices-dcut-cm} reports the number of feasible and optimal vertices of $P_{BCR}$ and $P_{CM}$, computed by complete enumeration, for instances with 4 and 5 nodes and 3 or 4 terminals, showing the potential impact of our approach on the overall number of vertices.

The core intuition of our new CM formulation is that in a complete metric graph, any Steiner node is visited only if its outdegree is at least twice the indegree, 
because if the indegree and outdegree of a Steiner node are both equal to 1, then an optimal solution with a smaller cost that avoids detouring in that node exists. 
Note that this requirement is profoundly different from asking a Steiner node not to be a leaf, which is a much weaker condition.
The existence of such a solution is guaranteed by the property that the graph is metric and complete. 
Indeed, such a solution may not exist in a non-complete graph.
Using these relations on the degree of Steiner nodes, we introduce a new family of constraints to the BCR formulation, reflecting our main intuition.

The main contributions of this paper are as follows:
\begin{enumerate}
\item We prove in \Cref{thm:21-metric} that our new polytope $P_{CM}$ contains among the integer vertices only those that could be optimal for metric costs. Given an integer vertex, we describe the cost vector that makes that vertex optimal.
\item We prove in \Cref{lemma:cm_connected} a connectedness property of the points of $P_{CM}$, and in \Cref{lemma:max_edges_cm} we set an upper bound on the number of edges in integer points.
\item  We characterize in \Cref{lemma:adding-zeros,lemma:vertex-iso1,lemma:vertex-iso2} isomorphic vertices of polytopes of different dimensions, namely, we link vertices of the polytope corresponding to the STP with $n$ nodes and $t$ terminals with the vertices corresponding to the STP with $n+1$ nodes and $t$ terminals, and vice versa.
\item Exploiting the previous results, we introduce two new heuristic algorithms for enumerating the vertices of $P_{CM}$.
Using these two algorithms, we compute vertices of $P_{CM}$ and $P_{BCR}$ with the largest known integrality gap for instances with up to 10 vertices.
\end{enumerate}
To the best of our knowledge, this work is the first attempt to extend the work of~\cite{boyd2007structure} to the metric Steiner Tree problem.
Note that an interesting STP instance with a small value of $n$ (i.e., $n=15$), but with a large integrality gap (i.e, equal to $\frac{8}{7}$) is the Skutella's graph shown in \Cref{fig:15-8} and reported in~\cite{konemann2011partitionbased}.
In our computational results, we will present other interesting instances with less than 10 nodes but having a large integrality gap, see \Cref{fig:max-and-diff-gaps}. 
\begin{figure}[t!]
\centering
\scalebox{0.5}{\begin{tikzpicture}
		\draw[line width=3, color=mycol4] (0,-2) -- (-6, -4) node [currarrow, pos=0.8, sloped, rotate=180] {};
		\draw[line width=3, color=mycol4] (0,-2) -- (-4, -4) node [currarrow, pos=0.8, sloped, rotate=180] {};
		\draw[line width=3, color=mycol4] (0,-2) -- (-2, -4) node [currarrow, pos=0.8, sloped, rotate=180] {};
		\draw[line width=3, color=mycol4] (0,-2) -- (0, -4) node [currarrow, pos=0.8, sloped] {};
		\draw[line width=3, color=mycol4] (0,-2) -- (2, -4) node [currarrow, pos=0.8, sloped] {};
		\draw[line width=3, color=mycol4] (0,-2) -- (4, -4) node [currarrow, pos=0.8, sloped] {};
		\draw[line width=3, color=mycol4] (0,-2) -- (6, -4) node [currarrow, pos=0.8, sloped] {};
		\draw[line width=3, color=mycol4] (-6, -4) -- (-6, -8) node [currarrow, pos=0.8, sloped] {};
		\draw[line width=3, color=mycol4] (-6, -4) -- (-4, -8) node [currarrow, pos=0.8, sloped] {};
		\draw[line width=3, color=mycol4] (-6, -4) -- (-2, -8) node [currarrow, pos=0.8, sloped] {};
		\draw[line width=3, color=mycol4] (-6, -4) -- (0, -8) node [currarrow, pos=0.8, sloped] {};
		\draw[line width=3, color=mycol4] (-4, -4) -- (-6, -8) node [currarrow, pos=0.8, sloped, rotate=180] {};
		\draw[line width=3, color=mycol4] (-4, -4) -- (-4, -8) node [currarrow, pos=0.8, sloped] {};
		\draw[line width=3, color=mycol4] (-4, -4) -- (2, -8) node [currarrow, pos=0.8, sloped] {};
		\draw[line width=3, color=mycol4] (-4, -4) -- (4, -8) node [currarrow, pos=0.8, sloped] {};
		\draw[line width=3, color=mycol4] (-2, -4) -- (-6, -8) node [currarrow, pos=0.8, sloped, rotate=180] {};
		\draw[line width=3, color=mycol4] (-2, -4) -- (-2, -8) node [currarrow, pos=0.8, sloped] {};
		\draw[line width=3, color=mycol4] (-2, -4) -- (2, -8) node [currarrow, pos=0.8, sloped] {};
		\draw[line width=3, color=mycol4] (-2, -4) -- (6, -8) node [currarrow, pos=0.8, sloped] {};
		\draw[line width=3, color=mycol4] (0, -4) -- (-6, -8) node [currarrow, pos=0.8, sloped, rotate=180] {};
		\draw[line width=3, color=mycol4] (0, -4) -- (0, -8) node [currarrow, pos=0.8, sloped] {};
		\draw[line width=3, color=mycol4] (0, -4) -- (4, -8) node [currarrow, pos=0.8, sloped] {};
		\draw[line width=3, color=mycol4] (0, -4) -- (6, -8) node [currarrow, pos=0.8, sloped] {};
		\draw[line width=3, color=mycol4] (2, -4) -- (-2, -8) node [currarrow, pos=0.8, sloped, rotate=180] {};
		\draw[line width=3, color=mycol4] (2, -4) -- (0, -8) node [currarrow, pos=0.8, sloped, rotate=180] {};
		\draw[line width=3, color=mycol4] (2, -4) -- (2, -8) node [currarrow, pos=0.8, sloped] {};
		\draw[line width=3, color=mycol4] (2, -4) -- (4, -8) node [currarrow, pos=0.8, sloped] {};
		\draw[line width=3, color=mycol4] (4, -4) -- (-4, -8) node [currarrow, pos=0.8, sloped, rotate=180] {};
		\draw[line width=3, color=mycol4] (4, -4) -- (-2, -8) node [currarrow, pos=0.8, sloped, rotate=180] {};
		\draw[line width=3, color=mycol4] (4, -4) -- (4, -8) node [currarrow, pos=0.8, sloped] {};
		\draw[line width=3, color=mycol4] (4, -4) -- (6, -8) node [currarrow, pos=0.8, sloped] {};
		\draw[line width=3, color=mycol4] (6, -4) -- (-4, -8) node [currarrow, pos=0.8, sloped, rotate=180] {};
		\draw[line width=3, color=mycol4] (6, -4) -- (0, -8) node [currarrow, pos=0.8, sloped, rotate=180] {};
		\draw[line width=3, color=mycol4] (6, -4) -- (2, -8) node [currarrow, pos=0.8, sloped, rotate=180] {};
		\draw[line width=3, color=mycol4] (6, -4) -- (6, -8) node [currarrow, pos=0.8, sloped] {};
		\fill[color=mycol1] (0,-2) circle (0.2);
		\fill[color=white] (0,-2) circle (0.08);
		\fill[color=mycol2] (-6.2, -4.2) rectangle (-5.8, -3.8);
		\fill[color=mycol2] (-4.2, -4.2) rectangle (-3.8, -3.8);
		\fill[color=mycol2] (-2.2, -4.2) rectangle (-1.8, -3.8);
		\fill[color=mycol2] (-0.2, -4.2) rectangle (0.2, -3.8);
		\fill[color=mycol2] (1.8, -4.2) rectangle (2.2, -3.8);
		\fill[color=mycol2] (3.8, -4.2) rectangle (4.2, -3.8);
		\fill[color=mycol2] (5.8, -4.2) rectangle (6.2, -3.8);
		\fill[color=mycol3] (-6,-8) circle (0.2);
		\fill[color=mycol3] (-4,-8) circle (0.2);
		\fill[color=mycol3] (-2,-8) circle (0.2);
		\fill[color=mycol3] (0,-8) circle (0.2);
		\fill[color=mycol3] (2,-8) circle (0.2);
		\fill[color=mycol3] (4,-8) circle (0.2);
		\fill[color=mycol3] (6,-8) circle (0.2);
\end{tikzpicture}}
\caption{Small STP instance with a large integrality gap: Skutella's graph with $n=15$, $t=8$, and $\alpha=\frac{8}{7}$~\cite{konemann2011partitionbased}. The hollow circle represents the root, the circles represent the terminals, and the squares represent the Steiner nodes. Every arc correspond to a variable $x_{ij}$ of value equal to $\frac{1}{4}$.}
\label{fig:15-8}
\end{figure}
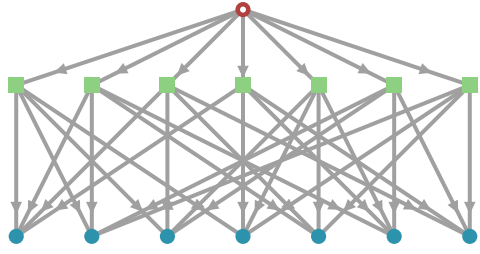

\begin{table}[t!]
\caption{Number of feasible and optimal vertices for $P_{BCR}$ and $P_{CM}$. While the BCR polytope has several (feasible) vertices that cannot be optimal for any metric cost, the CM polytope does not suffer this issue.}
\label{tab:vertices-dcut-cm}    
\centering
\begin{tabular}{ccrrcrr}
	\toprule
	&   & \multicolumn{2}{c}{$P_{BCR}$} && \multicolumn{2}{c}{$P_{CM}$} \\
	$n$ &  $t$  & feasible & optimal && feasible & optimal \\
	\midrule
	4 & 3 &  256 & 70 && 4 & 4\\
	5 & 3 &  28\,345 & 3\,655 && 5 & 5\\
	5 & 4 &  24\,297 & 3\,645 && 44 & 44\\
	\bottomrule
\end{tabular}

\end{table}

\paragraph{Outline}
The outline of this paper is as follows.
\Cref{sec:ig_dcut} reviews the approach introduced in~\cite{benoit2008finding,elliott2008integrality} for computing the integrality gap of small instances ($n \leq 10$) of the Traveling Salesman Problem and presents how to apply a similar approach to the BCR formulation of the STP. 
The critical step is the complete enumeration of the vertices of the BCR polytope, which are several thousand already for $n=5$, as shown in Table 1.
\Cref{sec:CM} presents the CM formulation, and we prove interesting properties of the corresponding polytope, which allows us to apply the methodology of~\cite{benoit2008finding,elliott2008integrality} to look for small instances of STP with a large integrality gap. 
In \Cref{sec:vertex_enum}, we observe that the exhaustive enumeration of vertices is intractable for $n \geq 6$, and we present two heuristic procedures for generating vertices of $P_{CM}$, exploiting graph isomorphism.
In \Cref{sec:results}, we present the vertices for graph $n \leq 10$, realizing the best-known lower bounds on the integrality gap for the CM and the BCR formulation. 
We conclude the paper with a perspective on future works.

\section{Integrality gap of BCR for fixed \texorpdfstring{$\bm n$}{n}}
\label{sec:ig_dcut}
In this section, we present our strategy to compute the integrality gap of the BCR formulation for STP, which is based on the approach introduced in~\cite{benoit2008finding,elliott2008integrality}. 

\subsection{Problem definition for computing the integrality gap of small instances}
Let $\mathcal{G}_{n, t} \coloneqq \{G \, | \, \text{metric with } n \text{ nodes and } T = [t]\}$. 
	We define 
	\begin{align}
		\alpha_{G} & := \dfrac{\text{STP}(G)}{\text{opt}_{BCR}(G)}\\
		\alpha_{n,t} & := \sup_{G \in \mathcal{G}_{n,t}} \alpha_{G}. 
\end{align}
Note that $\alpha_{G}$ is the integrality gap of a given instance of metric STP, while $\alpha_{n,t}$  is the maximum integrality gap once we fix both the cardinality of $T$ and $V$. 
Clearly, the integrality gap \eqref{GAP} is equal to
\begin{equation}
	\alpha = \sup_{n, t} \alpha_{n, t}.
\end{equation}
For $t = n$, $P_{BCR}$ is proven to be integral~\cite{edmonds1967optimum}, while~\cite{goemans1993catalog} shows the same result for $t = 2$, and hence $\alpha_{n, 2} = \alpha_{n, n} = 1$.
However, the exact value of $\alpha_{n, t}$ for $2 < t < n$ is unknown. 

In this work, we extend the approach presented in~\cite{benoit2008finding,elliott2008integrality} for TSP for computing the exact value of $\alpha$ for the BCR formulation.
Similarly to~\cite{benoit2008finding,elliott2008integrality}, if we divide the costs $c_{ij}, i, j \in V$ of an instance $G$ for the optimal value $\text{STP}(G)$, we obtain another instance $G'$, having an optimal value $\text{STP}(G)=1$ but the same set of optimal solutions.
Hence, defining $\mathcal{G}_{n, t}^1 \coloneqq \{G \, | \, \text{metric with } n \text{ nodes and } T = [t], \, \text{STP}(G) = 1\}$, we obtain
\[
\alpha_{n,t}  \coloneqq \sup_{G \in \mathcal{G}_{n, t}^1} \dfrac{1}{\text{opt}_{BCR}(G)}, 
\]
and hence
\begin{equation}
	\frac{1}{\alpha_{n,t}}  \coloneqq \inf_{G \in \mathcal{G}_{n, t}^1} \text{opt}_{BCR}(G). \label{eq:quadratic_compressed}
\end{equation}
As already mentioned, an instance $G$ with set of terminals $T \neq [t]$, can always be mapped into an equivalent graph with $T = [t]$ through a simple node relabeling. We will refer to both the instances and their resulting solutions as \emph{isomorphic}.

Note that \eqref{eq:quadratic_compressed} leads to an optimization problem having linear constraints but a quadratic objective function (note that $c_e$ and $x_{ij}$ are both decision variables) that can be formulated as follows:
\begin{subequations}\label{eq:quadratic_expanded}
	\begin{align}
		\min \quad & \sum_{\{i, j\} \in E} c_e (x_{ij} + x_{ji}) & \label{eq:quadratic_obj}\\    
		\mbox{s.t.} \quad    & x_{i j}+x_{j i} \leq 1, & e=\{i, j\} \in E,  \label{eq:quadratic_degree}\\
		& x\left(\delta^{-}(W)\right) \geq 1, & W \subset V \setminus \{1\}, \;W \cap \{1, \ldots, t\} \neq \emptyset,  \label{eq:quadratic_subtour} \\
		& 0 \leq x_{ij} \leq 1, & \forall i, j \in V, \; i \neq j, \label{eq:quadratic_bounds}\\
		& c_{ij} \geq 0, & \forall \{i, j\} \in E, \label{eq:quadratic_nn}\\
		& c_{ij} \leq c_{ik} + c_{jk}, &  \forall \{i, j\}, \{i, k\}, \{j, k\} \in E, \label{eq:quadratic_triangle_ineq} \\
		& \sum_{\{i, j\} \in E} c_e (\bar{z}_{ij} + \bar{z}_{ji}) \geq 1, & \forall \bar{z} \in S_{BCR}. \label{eq:quadratic_geq_1}
	\end{align}
\end{subequations}
Constraints \eqref{eq:quadratic_degree}--\eqref{eq:quadratic_bounds} ensures the feasibility of vector $x$, while constraints \eqref{eq:quadratic_nn}--\eqref{eq:quadratic_triangle_ineq} ensure the property of $c$ being metric. Constraint \eqref{eq:quadratic_geq_1} enforce that for every feasible integer solution, the cost is at least $1$. This ensures that the optimum is exactly $1$, that is, $\text{STP}(G) = 1$.
Our preliminary computational results show that \eqref{eq:quadratic_obj}--\eqref{eq:quadratic_geq_1} are intractable even for $n = 5$ and $t=3, 4$.
Therefore, we proceed as done in~\cite{benoit2008finding,elliott2008integrality}, exploiting the vertex representation of $P_{BCR}$.
Similarly to~\cite{benoit2008finding,elliott2008integrality}, and by recalling that for each cost vector $c$, there exists an optimal solution attained at a vertex, we can re-write \eqref{eq:quadratic_obj}--\eqref{eq:quadratic_geq_1} as a \emph{linear program} for each vertex $\overline{x}$ of $P_{BCR}(n,t)$.
\begin{subequations}\label{eq:linear}
	\begin{align}
		\min  \quad & \sum_{\{i, j\} \in E} c_e (\overline{x}_{ij} + \overline{x}_{ji}) & \\   
		\mbox{s.t.} \quad  & c_{ij} \geq 0, & \forall \{i, j\} \in E,\\
		& c_{ij} \leq c_{ik} + c_{jk}, &  \mathllap{\forall \{i, j\}, \{i, k\}, \{j, k\}   \in E,}\\
		& \text{the optimal solution of } c \text{ is attained at } \overline{x}, & \label{eq:linear_stay_here} \\
		& \sum_{\{i, j\} \in E} c_e (\bar{z}_{ij} + \bar{z}_{ji}) \geq 1, & \forall \bar{z} \in S_{BCR}. \label{eq:linear_geq_1}
	\end{align}
\end{subequations}
and define $\text{Gap}(\overline{x})$ as the optimal value of \eqref{eq:linear}.

As in~\cite{benoit2008finding,elliott2008integrality}, we observe that constraint \eqref{eq:linear_stay_here} can be formulated using the complementary slackness conditions. 
Such conditions ensure that $\overline{x}$ belongs to the points minimizing the STP with cost $c$. 
A detailed formulation is presented below.
	\begin{subequations}
		\begin{align}
			\label{eq:gap_dcut_ibj} \text{(Gap)} \quad    \min \quad  & \sum_{\{i, j\} \in E} c_e (\overline{x}_{ij} + \overline{x}_{ji})  \\
			\mbox{s.t.} \quad & c_{ij} \leq c_{ik} + c_{jk}, &  \forall \{i, j\}, \{i, k\}, \{j, k\}   \in E, \\
			&  y_e + \sum_{(i, j) \in \delta^-(W)} z_W + d_{ij} \leq c_e, &  \forall  (i, j) \in A, \\
			& y_e = 0, &   \mathllap{\forall(i, j) \text{ s.t. }\overline{x}_{ij}+\overline{x}_{j i} < 1, \; e=\{i, j\} \in E,} \\
			&  z_W  = 0, &  \mathllap{\forall W \subset V \setminus \{r\}, \,  W \cap T \neq \emptyset \text{ s.t. } \sum_{(i, j) \in \delta^-(W)} \overline{x}_{ij} > 1,} \\
			&  d_{ij} = 0, &   \forall (i, j) \in A \text{ s.t }\overline{x}_{ij} = 1, \\
			& y_e + \sum_{(i, j) \in \delta^-(W)} z_W + d_{ij} - c_e = 0, & \forall (i, j) \in A \text{ s.t } \overline{x}_{ij} > 0, \\
			& \sum_{\{i, j\} \in E} c_e (\bar{z}_{ij} + \bar{z}_{ji}) \geq 1, & \forall \bar{z} \in S_{BCR}(n,t),\label{eq:gap_dcut_geq_1}\\
			& y_e, d_{ij}, d_{ji}\leq 0, &  \forall e=\{i, j\} \in E, \\
			&  z_W \geq 0, &   \mathllap{\forall W \subset V \setminus \{r\}, \,  W \cap T \neq \emptyset,} \\
			&  c_{ij} \geq 0, &  \forall \{i, j\}   \in E.  \label{eq:gap_dcut_c_bound}   
		\end{align}
	\end{subequations} 

\section{A novel formulation for the complete metric case}
\label{sec:CM}
Scip-Jack~\cite{gamrath2017scip}, the state-of-art solver for the (Graphic) STP, relies on a modified version of the BCR formulation, firstly introduced in~\cite{Koch1998}.
The formulation used by Scip-Jack (SJ) is presented through its associated polytope:

\begin{subequations}\label{eq:psj}
		\begin{align}
			P_{SJ}(n,t) \coloneqq \{& x \in [0,1]^m \, : \\    
			& x\left(\delta^{-}(W)\right) \geq 1, &  W \subset V \setminus \{r\}, \;W \cap T \neq \emptyset, \label{eq:scip_subt}\\ 
			& x\left(\delta^{-}(r)\right) = 0, &  \label{eq:scip_root_inflow}\\
			& x\left(\delta^{-}(v)\right) = 1, & v \in T \setminus \{r\}, \label{eq:scip_inflow_t}\\
			& x\left(\delta^{-}(v)\right) \leq  1, & v \in V \setminus T,  \label{eq:scip_inflow_s}\\
			& x\left(\delta^{-}(v)\right) \leq x\left(\delta^{+}(v)\right), & \forall v \in V \setminus T, \label{eq:scip_inout1}\\ 
			& x\left(\delta^{-}(v)\right) \geq x_a, & \forall a \in \delta^{+}(v), v \in V \setminus T \label{eq:scip_inout2} \},
		\end{align}
\end{subequations}
where $\delta^{+}(W):= \{(i, j) \; \vert \; i  \in W, \; j \not \in W \}$. 

Constraints \eqref{eq:scip_root_inflow}--\eqref{eq:scip_inflow_s} describe the inflow of every node: the first equation ensures that no inflow is present in the root, the second equation ensures that the inflow of terminal nodes is exactly equal to $1$, since every terminal must be reached, and the third equation ensures that the inflow of non-terminal nodes is smaller or equal than $1$ since a non-terminal node may or may not be part of an optimal solution. Note that both terminal and non-terminal nodes have an inflow of at most $1$ so that at most one path exists from the root to any node. Constraint \eqref{eq:scip_inout1} ensures that non-terminal nodes cannot be leaves of the solution. Constraint \eqref{eq:scip_inout2} ensures that no flow is generated from non-terminal nodes.
Notice that this formulation is not specific to the metric case we want to attack, as illustrated by the example below.
\begin{example}\label{ex:not_connected}
	Let $G = (V, E)$ be a complete metric graph with $V = \textcolor{blue}{[5]}$ and let $T = \textcolor{blue}{[2]}$. Define $x$ as the following
	\begin{equation}
		x_{ij} =
		\begin{cases}
			1,  & \text{if } (i,j) \in \{(1, 2), (3, 4), (4, 5), (5, 3)\},\\
			0, & \text{else.} 
		\end{cases}
	\end{equation}
	We have that $x$ is feasible for the SJ formulation with $r = 1$. In particular, every Steiner node has degree $2$, since it has indegree and outdegree $1$. However, this solution is never optimal for any metric cost since by setting $x_{3,4} = x_{4,5} = x_{5, 3} = 0$ we obtain a feasible solution with a strictly smaller cost. Note that, in particular, $x$ is not connected.
\end{example}
To prevent this issue, we introduce a stronger formulation tailored to the Complete Metric (CM) case, which is presented below through its associated polytope
\begin{subequations}\label{eq:pcm}
		\begin{align}
			P_{CM}(n,t) \coloneqq \{& x \in [0,1]^m \, : \\    
			& x\left(\delta^{-}(W)\right) \geq 1, &  W \subset V \setminus \{r\}, \;W \cap T \neq \emptyset, \label{eq:cm_subt}\\ 
			& x\left(\delta^{-}(r)\right) = 0, \label{eq:cm_root_inflow}\\
			& x\left(\delta^{-}(v)\right) \leq  1, & v \in V \setminus \{r\}, \label{eq:cm_inflow}\\
			& 2x\left(\delta^{-}(v)\right) \leq x\left(\delta^{+}(v)\right), & v \in V \setminus T \label{eq:cm_in_out_flow} \}.
		\end{align}
\end{subequations}
In particular, in our new formulation, the left-hand side of Constraint \eqref{eq:scip_inout1} is multiplied by $2$. This ensures that a non-terminal node has an outflow that is at least twice the inflow. Hence, for the integer solutions, a non-terminal node is visited only if its outflow is at least $2$. The idea is that, in a complete metric graph, if the inflow and the outflow of a non-terminal node are both equal to $1$, then there exists an optimal solution with a smaller cost that avoids detouring in that node. The existence of such a solution is guaranteed by the property that the graph is metric and complete. 
Note that such a solution may not exist in a non-complete graph, for example, when $G = (V, E)$ with $V = \{1, 2, 3\}$, $E = \{\{1,3\}, \{2, 3\}\}$ and $T = \{1, 2\}$.
We also avoid adding the equivalent of Constraint \eqref{eq:scip_inout2} because of the following lemma. 

\begin{lemma}
	\label{lemma:constr-redundancy}
	When dealing with positive costs, Constraint \eqref{eq:scip_inout2} is redundant even for the simpler BCR formulation.
\end{lemma}
A proof of this fact is contained in \Cref{app:lemma-proof}.

Before discussing how we use this formulation to retrieve information regarding the integrality gap of the BCR formulation, we list some properties of the CM formulation that we consider interesting.

\subsection{Properties of the complete metric formulation}

We first show that with mild hypothesis the set of integer solutions of the SJ formulation coincides with the set of integer solutions of the CM formulation.

\begin{lemma}\label{lemma:cm_solutions}
	Let $G$ be a complete metric graph with $c \in \mathbb{R}^{(n-1) \times n/2}$ defining the edge weights and let $x\in S_{SJ}$ be optimal for the cost vector $c$. If
	\begin{equation}\label{eq:strict_cost}
		c_{ij} < c_{ik} + c_{kj} \quad \forall \{i, j\}, \{i, k\}, \{j, k\}   \in E,
	\end{equation}then $x \in S_{CM}$ and it is optimal for the same cost vector. Moreover, if $y \in S_{CM}$ is optimal for $G$, then $y \in S_{SJ}$ and it is optimal for $G$.
\end{lemma}
\begin{proof}
	First, we prove $x \in S_{CM}$. We clearly have that $S_{CM} \subset S_{SJ}$. Suppose then by contradiction that there exists $x \in S_{SJ}$, optimal for $G$ such that $x \notin S_{CM}$. Because of the constraints that describe the two models, there exists $j\in V \setminus T$
	\begin{align*}
		\sum_{i \neq j} x_{ij} &\leq \sum_{k \neq j} x_{jk} && \text{Feasible for \eqref{eq:psj}} \\
		2 \cdot \sum_{i \neq j} x_{ij}  &> \sum_{k \neq j} x_{jk} && \text{Infeasible for \eqref{eq:pcm}}.
	\end{align*}
	Combining the two inequalities and the integrality of $x$, it holds that there exist $i \in V$ such that $x_{ij} = 1$. Using \eqref{eq:scip_inout2}, it follows that it exists $k \in V$ such that $x_{jk} = 1$ . Since we are in a complete metric graph, setting these two variables to zero and setting $x_{ik}= 1$ gives us a feasible solution, which is also of smaller cost because of hypothesis \eqref{eq:strict_cost}, which is in contradiction with the optimality of $x$. The optimality follows from the optimality for SJ.
	
	For the second statement, let $y \in S_{CM}$ optimal for $G$. Clearly, $y \in S_{SJ}$. Suppose by contradiction that there exists $z \in S_{SJ}$ such that $c^\intercal z < c^\intercal y$. For the first part of the proof, we have that $z \in S_{CM}$ optimal for CM, which contradicts the optimality of $y$.
\end{proof}

\begin{observation}\label{obs:cm_exist}
	Note that, without hypothesis \eqref{eq:strict_cost}, we can say that given a complete metric graph $G$ and $x \in S_{SJ}$ optimal for $G$, there exists $x'\in S_{SJ}$ optimal for $G$ such that $x' \in S_{CM}$  optimal for $G$. In particular, $x'$ is chosen as one of the optimal integer solutions of SJ that avoids detouring into non-terminal nodes, where detouring into a node means entering with one edge and exiting with one edge. The same reasoning can be applied to the BCR formulation.
\end{observation}

\begin{observation}
	Note that \Cref{lemma:cm_solutions} does not hold when removing the integrality hypothesis. Take, for example, as graph $G$ the metric completion of the instance \texttt{se03} of the SteinLib~\cite{Koch2001}, having 13 nodes and 7 terminals. We have that
	\begin{equation*}
		\text{opt}_{SJ}(G) = 11 < 12 =  \text{opt}_{CM}(G) =  \text{STP}(G).
	\end{equation*}
	We then have that $\text{opt}_{SJ}(\cdot) \leq \text{opt}_{CM}(\cdot)$, and so the integrality gap of the CM formulation is a lower bound for the integrality gap of the SJ formulation on complete metric graphs. Moreover, the bound is not always tight. The same holds for the BCR formulation.
\end{observation}

A concept we will use extensively in the following is the one of \emph{support graph}
	\begin{definition}[Support graph]
		Given $x\in P_{\mathcal{F}}, \; \mathcal{F}$ being one among CM, SJ, BCR, or any other valid formulation for the ST, let $G_x = (V_x, E_x)$ denote the corresponding weighted support (di)graph, that is $V_x = \{i \in V \, : \, x(\delta^+(i)) + x(\delta^-(i)) > 0\}$ and $E_x = \{e = \{i,j\} \in E  \, : \, x_{ij} + x_{ji} > 0\}$, and weight equal to $x_{ij} \; \forall i\neq j \in V$. 
\end{definition}

An interesting property of the CM formulation is connectedness. Constraints \eqref{eq:scip_subt} enforce that in an SJ solution, all the terminal nodes belong to the same connected component, but this is not guaranteed for non-terminal nodes. For the CM formulation instead, the following lemma holds.

\begin{lemma}\label{lemma:cm_connected}
	The support graph of any feasible point of $P_{CM}$ is connected.
\end{lemma}

\begin{proof}
	It suffices to prove that no connected components without terminals exist since every terminal belongs to the same connected component because of Constraint \eqref{eq:cm_subt}. Let $x \in P_{CM}$ and let $H \subset V\setminus T$ be a subset of vertices inducing a connected component of $G_x$. Constraint \eqref{eq:cm_in_out_flow} implies
	\begin{equation*}
		\sum_{i, j \in H} x_{ij} = \sum_{j \in H} \sum_{i \in H} x_{ij} \geq \sum_{j \in H} 2 \sum_{i \in H} x_{ji},
	\end{equation*}
	because $H$ induces a connected component, so there is no inflow / outflow in / from $V \setminus H$. Re-arranging everything, we have
		\begin{equation*}
			\sum_{j \in H} 2 \sum_{i \in H} x_{ji} = 2 \sum_{i, j \in H} x_{ji} = 2\sum_{i, j \in H} x_{ij}.
	\end{equation*}
	The only possibility is that $\sum_{i, j \in H} x_{ij} = 0$, so no connected component without terminals can be part of a feasible solution for CM.
\end{proof}

Note that \Cref{lemma:cm_connected} does not hold for the SJ formulation, as \Cref{ex:not_connected} shows.

Another interesting property of the CM formulation deals with constraint reduction. In this case, we can prove theoretical results on the number of edges in an integer CM solution and, consequently, on the number of Steiner nodes.
\begin{lemma}\label{lemma:max_edges_cm}
	Let $x \in S_{CM}(n,t)$. Then, $x$ verifies
	\begin{equation}\label{eq:max_edges_cm}
		\sum_{i, j \in V, i \neq j} x_{ij} \leq \min(n-1, 2t-3).        
	\end{equation}
\end{lemma}
\begin{proof}
	We know that $G_x$ is acyclic because of Constraint \eqref{eq:cm_inflow} and \ref{eq:cm_root_inflow} and we also know that $G_x$ is connected because of \Cref{lemma:cm_connected}, so $G_x$ is a tree. Since $|V_x| \leq n$, we have that $\sum_{i, j} x_{ij} \leq n-1$.
	Now, we only need to prove that $\sum_{i, j} x_{ij} \leq 2t-3.$ We have that
	\begin{align*}
		\sum_{i, j \in V, i \neq j} x_{ij} = \sum_j \sum_{i \neq j} x_{ij} &=
		\sum_{i \neq r} x_{ir} + \sum_{j \in T \setminus \{r\}} \sum_{i \neq j} x_{ij} + \sum_{j \in V \setminus T} \sum_{i \neq j} x_{ij} = \\
		& \leq  0 + (t-1) + \frac{1}{2} \sum_{j \in V \setminus T} \sum_{k \neq j} x_{jk},
	\end{align*}
	where the last inequality holds, because of Constraint \eqref{eq:cm_root_inflow}, Constraint \eqref{eq:cm_inflow}, and Constraint \eqref{eq:cm_in_out_flow}, respectively. Note that only the last one gives us inequality since the others hold with equality. We can now rewrite
	\begin{equation*}
		\sum_{j \in V \setminus T} \sum_{k \neq j} x_{jk} = \sum_{i, j} x_{ij} - \sum_{j \in T} \sum_{k \neq j} x_{jk}.
	\end{equation*}
	Combining this with the previous equation, we get that
	\begin{equation*}
		\sum_{i, j} x_{ij} \leq t - 1 + \frac{1}{2} \sum_{i, j} x_{ij} - \frac{1}{2} \sum_{j \in T} \sum_{k \neq j} x_{jk}. 
	\end{equation*}
	Rearranging the terms, we obtain
	\[
	\frac{1}{2}  \sum_{i, j} x_{ij} \leq t - 1 - \frac{1}{2} \sum_{j \in T} \sum_{k \neq j} x_{jk}
	\]
	and hence, multiplying by 2
	\begin{align*}
		\sum_{i, j} x_{ij} & \leq 2t -  2 -  \sum_{j \in T} \sum_{k \neq j} x_{jk} = \\
		& = 2t - 2 - \sum_{k \neq r} x_{rk} - \sum_{j \in T \setminus \{r\}} \sum_{k \neq j} x_{jk} \leq 2t - 2 - 1 - 0 = 2t -3,
	\end{align*}
	where the last inequality holds because $\sum_{k \neq r} x_{rk} \geq 1$ by taking $W = V \setminus \{r\}$ in Constraint \eqref{eq:cm_subt}, and because $x_{jk} \geq 0$, respectively.
\end{proof}

\begin{observation}
	Let $t \leq \frac{n}{2}  +1$ and so $\min (n-1, 2t-3) = 2t-3$. Then, if we consider the CM, our solution is a tree with at most $2t-3$ edges, so it has $2t-3+1 = 2t-2$ nodes, $t$ of which are terminals, leaving us with $t-2$ Steiner vertices. Thus, it suffices to write Constraints \eqref{eq:cm_subt} only for 
	\begin{equation}
		W = W_1 \sqcup W_2, \quad W_1 \subset T \setminus {r}, \; |W_1| \geq 1, \quad W_2 \subset V \setminus T, \; |W_2| \leq t-2,
	\end{equation}
	instead of writing it for any $W = W_1 \sqcup W_2$, $W_2 \subset V \setminus T$. For instance, in the case $(n, t) = (20, 5)$ we go from $(2^4 -1)\times 2^{15} = 491520$ possible choices of $W$ to just $(2^4 -1)\times \sum_{i=0}^3 \binom{15}{i} = 8640$, which is around $1.8\%$ of the total.
\end{observation}

\begin{observation}
	Note that the bound proved in \Cref{lemma:max_edges_cm} is tight, that is, if we have $t \leq \frac{n}{2} +1$, there exists $x \in S_{CM}(n,t)$ such that $\sum_{i,j} x_{ij} = 2t -3$. If we have $T = [t], \, V = [n], \, \{r\} = \{1\}$, $x$ is as follows:
		\begin{align*}
			x_{1,t+1} = x_{2t-2, t} &= 1, &&\\
			x_{t+i, t+i+1} &= 1,  && i = 1, \dots, t-3, \\
			x_{t+i, i+1} &= 1, &&i = 1, \dots, t-2,
		\end{align*}
		with a graphical representation being
		\begin{center}
			\scalebox{0.5}
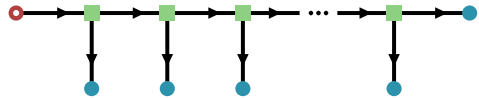
{
				\begin{tikzpicture}
					\draw[line width=3, color=black] (0,0) -- (2, 0) node [currarrow, pos=0.6] {};
					\draw[line width=3, color=black] (2,0) -- (4, 0) node [currarrow, pos=0.6] {};
					\draw[line width=3, color=black] (4,0) -- (6, 0) node [currarrow, pos=0.6] {};
					\draw[line width=3, color=black] (6,0) -- (8, 0) node [currarrow, pos=0.4] {};
					\draw[line width=3, color=black] (8,0) -- (10, 0) node [currarrow, pos=0.6] {};
					\draw[line width=3, color=black] (10,0) -- (12, 0) node [currarrow, pos=0.6] {};
					\draw[line width=3, color=black] (2,0) -- (2, -2) node [currarrow, pos=0.6, sloped] {};
					\draw[line width=3, color=black] (4,0) -- (4, -2) node [currarrow, pos=0.6, sloped] {};
					\draw[line width=3, color=black] (6,0) -- (6, -2) node [currarrow, pos=0.6, sloped] {};
					\draw[line width=3, color=black] (10,0) -- (10, -2) node [currarrow, pos=0.6, sloped] {};
					\fill[color=mycol1] (0,0) circle (0.2);
					\fill[color=white] (0,0) circle (0.08);
					\fill[color=mycol2] (1.8, -0.2) rectangle (2.2, 0.2);
					\fill[color=mycol2] (3.8, -0.2) rectangle (4.2, 0.2);
					\fill[color=mycol2] (5.8, -0.2) rectangle (6.2, 0.2);
					\fill[color=mycol2] (9.8, -0.2) rectangle (10.2, 0.2);
					\fill[color=mycol3] (12,0) circle (0.2);
					\fill[color=mycol3] (2,-2) circle (0.2);
					\fill[color=mycol3] (4,-2) circle (0.2);
					\fill[color=mycol3] (6,-2) circle (0.2);
					\fill[color=white] (8,0) circle (0.5);
					\fill[color=black] (8,0) circle (0.06);
					\fill[color=black] (7.8,0) circle (0.06);
					\fill[color=black] (8.2,0) circle (0.06);
					\fill[color=mycol3] (10,-2) circle (0.2);
				\end{tikzpicture}
			}
		\end{center}
		where $t-2$ Steiner nodes belongs to $x$. The outflow of every terminal that is not the root is $0$, the outflow of the root is $1$, while the outflow of every Steiner node is $2$ and so $\sum_{i,j} x_{ij} = 1+ 2 \times (t-2) = 2t-3$.
\end{observation}

After discussing the properties that make the CM formulation interesting by itself, we now focus on commenting on its advantages in deducing information on the lower bounds of the BCR.

First, we discuss why studying the complete metric case is not restrictive. In particular, we use the \emph{metric closure} of a graph, defined below. 
\begin{definition}[Metric closure of a graph]
	Let $G = (V, E)$ be an  connected  edge-weighted graph. We define the \emph{metric closure} of $G$ the complete metric graph $\bar{G} = (V, \bar{E})$ such that the weight of the edge $\{i,j\}$ in $\bar{G}$ is equal to the value of the shortest paths from $i$ to $j$ in $G$. 
\end{definition}
We now link the integrality gap of the BCR formulation of a graph to the corresponding integrality gap of its metric closure.
\begin{lemma}\label{lemma:metric-closure-gap}
	Let $G = (V, E)$, $T \subset V$ be a Steiner instance, and let $\bar{G}$ be the Steiner instance corresponding to the metric closure of $G$. Then we have that
	\begin{equation}
		\text{STP}(G) = \text{STP}(\bar{G}), \quad \text{opt}_{BCR}(G) = \text{opt}_{BCR}(\bar{G}).
	\end{equation}
\end{lemma}
\begin{proof}
	Let $x$ be an integer feasible solution for $G$ for the BCR formulation. Then, it is also a feasible solution for $\bar{G}$, and because of the definition of metric closure, it is a feasible solution with a smaller cost. We have then that $\text{STP}(\bar{G}) \leq \text{STP}(G)$. Let now $\bar{x}$ be a feasible solution for $\bar{G}$ for the BCR formulation. Reasoning in a non-oriented way, if we take every edge of $\bar{x}$ and substitute it with the corresponding shortest path in $G$, we obtain a subgraph of $G$ that can be oriented as a feasible solution $x$ of $G$, with a smaller cost. The cost is (non-strictly) smaller because we may take the same edge in different shortest paths. We then have that $\text{STP}(\bar{G}) \geq \text{STP}(G)$ and so $\text{STP}(\bar{G}) = \text{STP}(G)$.
	
	For the same reasoning, we have that $\text{opt}_{BCR}(G) = \text{opt}_{BCR}(\bar{G})$, with the exception that, when substituting an edge of $\bar{G}$ with the corresponding shortest path in $G$, since we are dealing with fractional solutions, if we have to take the same edge multiple times because it appears in multiple shortest paths, we have to take the minimum between $1$ and the sum of all the values with which that edge appears. This choice preserves feasibility and does not produce a solution with a larger cost.
\end{proof}
From this lemma, it follows that the integrality gap calculated with respect to the CM formulation only on metric graphs is a lower bound for the integrality gap of the BCR formulation across all graphs.

\subsection{The gap problem for the CM formulation}
With this in mind, one can proceed as in \Cref{sec:ig_dcut} and define a Gap problem for the CM formulation. 
Given $\overline{x}$ vertex of $P_{CM}(n,t)$, we define its gap as the linear problem of finding the cost vector that maximizes the integrality gap of a vertex $\overline{x}$, among those for which $\overline{x}$ is optimal. 

Given a vertex $x \in P_{CM}(n,t)$, we introduce variables $c_{ij}, \; \{i, j\} \in E$ that satisfy the triangle inequality and non-negativity constraints. Adding them to the slackness compatibility conditions we obtain the following linear program with exponentially many variables and constraints:
\begin{subequations} \label{eq:gap-cm}
	\begin{align}
		\min \quad  & \sum_{\{i, j\} \in E} c_e (\overline{x}_{ij} + \overline{x}_{ji})  \\
		\mbox{s.t.} \quad & c_{ij} \leq c_{ik} + c_{jk} &  \forall \{i, j\}, \{i, k\}, \{j, k\}   \in E,\\
		& y_{ij} + v_j + \sum_{w \in W(ij)} z_w  - c_{ij} = 0, & \mathllap{\forall i \in T, j \in T \setminus \{r\}, \; \bar{x}_{ij} > 0,}\\ 
		& y_{ij} + v_j + 2u_j + \sum_{w \in W(i,j)} z_w - c_{ij} = 0, & \mathllap{\forall i \in T, \; j \in V \setminus T, \; \bar{x}_{ij} > 0,}\\
		& y_{ji} + v_i - u_j + \sum_{w \in W(j,i)} z_w - c_{ji} = 0,  & \mathllap{\forall i \in  T \setminus \{r\}, j\in V \setminus T, \; \bar{x}_{ij} > 0,}\\
		& y_{ij} + v_j +2 u_j - u_i + \sum_{w \in W(i,j)} z_w - c_{ij} = 0,  & \forall i, j \in V \setminus T, \; \bar{x}_{ij} > 0,\\
		& z_W = 0, & \mathllap{\forall  W \subset V \setminus \{r\}, \;W \cap T \neq \emptyset \; \bar{x}\left(\delta^{-}(W)\right) > 1,}\\
		& v_j = 0, & \mathllap{\forall j \in  V \setminus \{ r\}, \; \bar{x}\left(\delta^{-}(j)\right) < 1,}\\
		& u_j = 0, &\mathllap{\forall j \in V \setminus T, \; 2\bar{x}\left(\delta^{-}(v)\right) < \bar{x}\left(\delta^{+}(v)\right),}\\
		& y_{ij} = 0, & \forall (i, j) \in A, \; \bar{x}_{ij} < 1,\\
		& c_{ij} \geq  0, & \forall \{i, j\}   \in E,\\
		& \sum_{\{i, j\} \in E} c_e (\bar{z}_{ij} + \bar{z}_{ji}) \geq 1, & \forall \bar{z} \in S_{CM}(n,t), \\
		& y_e \leq 0, &  \forall e=\{i, j\} \in E, \\
		& v_i, u_j \leq 0, &  \forall i \in V \setminus{\{r\}}, j \in V \setminus T, \\
		&  z_W \geq 0, &  \mathllap{\forall W \subset V \setminus \{r\}, \,  W \cap T \neq \emptyset.}
	\end{align}
\end{subequations}

\section{Properties and partial enumeration of nontrivial vertices}
\label{sec:vertex_enum}
In this section, we present the theoretical results and algorithms used to enumerate vertices of the polytope $P_{CM}(n,t)$. 
We first introduce results linking polytopes of different dimensions, and then, relying upon these and other structural results, we present two different algorithms for vertices enumeration.

\subsection{Avoiding redundancy}
First, let us define a particular class of vertices that will be of interest for our results.

\begin{definition}[Spanning vertex]\label{def:spanning-vert}
	Let $x$ be a vertex of $P_{CM}(n,t)$. We will call $x$ a \emph{spanning vertex} if all of the nodes are part of the solution $x$, that is, $x(\delta^-(i)) + x(\delta^+(i)) > 0$ for all $i \in V$.
\end{definition}
Note that \Cref{lemma:cm_connected} implies that every spanning vertex is connected.
In an STP, Steiner nodes may or may not be part of an optimal solution. This holds for vertices of $P_{CM}(n,t)$, both integer and non-integer, that is, not all of the vertices are spanning vertices. 
Hence, we can consider whether a non-spanning vertex of $P_{CM}(n,t)$ can be seen as a spanning vertex of a polytope of a smaller dimension, and vice versa, that is, if a spanning vertex of $P_{CM}(n,t)$ can be seen as a vertex of a polytope of a larger dimension.
The following results link vertices of $P_{CM}(n+1,t)$ with vertices of $P_{CM}(n,t)$ and vice versa. These results will be used in the enumeration of vertices to reduce the dimension of our search space by avoiding redundancy.

\begin{lemma}\label{lemma:adding-zeros}
	Let $x \in \mathbb{R}^{(n-1) \times n}$. Define $y \in \mathbb{R}^{n \times (n+1)}$ as
	\begin{equation}\label{eq:add-zeros}
		y_{ij} =
		\begin{cases}
			x_{ij},  &  1 \leq  i, j < n+1,\\
			0, & \text{otherwise}. 
		\end{cases}
	\end{equation}
	Then, $x \in P_{CM}(n,t)$ if and only if $y \in P_{CM}(n+1,t)$.
\end{lemma}

\begin{proof}
	Let $x \in P_{CM}(n,t)$. Note that $y \in [0,1]^{n \times (n+1)}$ since $x \in P_{CM}(n,t) \subset [0,1]^{(n-1) \times n}$. Regarding Constraint \eqref{eq:cm_subt}, we distinguish two cases. Let $W$ be a set as described in \eqref{eq:cm_subt} for $y$. ($a$) If $n+1 \in W$, going from $x$ to $y$ adds the variables $x_{i,n+1}$ which are all zero so since $x$ satisfies the constraint $y$ satisfies it too. ($b$) If $n+1 \notin W$, going from $x$ to $y$ adds the variables $x_{n+1,j}$ which are all zero so since $x$ satisfies the constraint $y$ satisfies it too. Constraints \eqref{eq:cm_root_inflow}--\eqref{eq:cm_inflow} are satisfied by $y$ since $x$ satisfies them, and we are only adding variables that take the value zero. Regarding Constraint \eqref{eq:cm_in_out_flow}, if $j = n+1$, the constraint holds trivially since all the variables are zero. If $j \neq n+1$, going from $x$ to $y$ adds the variables $x_{i,n+1}, x_{n+1,j}$ which are all zero, so since $x$ satisfies the constraint, $y$ also satisfies it.
	
	Let $y \in P_{CM}(n+1,t)$ of the form \eqref{eq:add-zeros}. Note that $x \in [0,1]^{n \times (n-1)}$ since $y \in P_{CM}(n+1,t) \subset [0,1]^{(n+1) \times n}$. Let $W$ be a set as described in \eqref{eq:cm_subt} for $x$. Let $\hat{W} \coloneqq W \cup \{n+1\}$. $\hat{W}$ is a set for which $y$ satisfies the correspondent constraint \eqref{eq:cm_subt}. In the $\hat{W}$ constraint, the only variables that appear are the ones appearing in the $W$ constraint plus the variables $x_{i,n+1}$ which are all zero. Since the $\hat{W}$ constraint is satisfied by $y$, the $W$ constraint is satisfied by $x$. Constraints \eqref{eq:cm_root_inflow}--\eqref{eq:cm_inflow} are clearly satisfied by $x$ since $y$ satisfies them. Regarding Constraint \eqref{eq:cm_in_out_flow}, passing from $y$ to $x$ removes the variables $x_{i,n+1}, x_{n+1, j}$ which are all zero, so since $y$ satisfies the constraint, $x$ also satisfies it.
\end{proof}

The following lemmas show how to identify vertices of $P_{CM}(n+1,t)$ with the ones of $P_{CM}(n,t)$ and vice versa.
\begin{lemma}\label{lemma:vertex-iso1}
	Let $x$ be a vertex of $P_{CM}(n,t)$. Then
	\begin{equation}
		y_{ij} =
		\begin{cases}
			x_{ij},  & \text{if } i, j \neq n+1,\\
			0, & \text{otherwise} 
		\end{cases}
	\end{equation}
	is a vertex of $P_{CM}(n+1,t)$.
\end{lemma}
\begin{proof}
	The idea of the proof is to show by contradiction that if $y$ is not a vertex, then $x$ cannot be a vertex as well. In detail, we have that $y \in P_{CM}(n+1,t)$ because of \Cref{lemma:adding-zeros}. Let $ P_{CM}(n+1,t)_0$ be the subpolytope of $P_{CM}(n+1,t)$ defined as
	\begin{equation*}
		P_{CM}(n+1,t)_0 \coloneqq \{z \in P_{CM}(n+1,t) \, \colon \, z_{i, n+1} = z_{n+1, j} = 0, \; 1 \leq i,j \leq n\}.
	\end{equation*}
	Let
	\begin{align*}
		\begin{aligned}
			\pi \colon P_{CM}(n+1,t)_0&\twoheadrightarrow P_{CM}(n,t)\\
			(z_{ij})_{i, j} &\mapsto (z_{ij})_{i, j \neq n+1}        
		\end{aligned}
	\end{align*}
	be the projection considering the first $n$ nodes. Note that $\pi  (y) = x$ and that $\pi$ is an injective map. Note also that $\text{Im}(\pi) \subset P_{CM}(n,t)$ because of \Cref{lemma:adding-zeros}.
	By contradiction, suppose that there exist $a, b \in P_{CM}(n+1,t)$ such that $a  \neq b$, $y = \frac{1}{2} a + \frac{1}{2} b$. Given $i, j \leq n$, we have that
	\begin{equation*}
		y_{i, n+1} = y_{n+1, j} = 0 = \frac{1}{2} (a_{i, n+1} + b_{i, n+1}) = \frac{1}{2} (a_{n+1, j} + b_{n+1, j}). 
	\end{equation*}
	Since $a, b \in P_{CM}(n+1,t)$, we have that $a_{i, n+1},b_{i, n+1},  a_{n+1, j},b_{n+1, j} \geq 0$ and so $a_{i, n+1},b_{i, n+1},  a_{n+1, j},b_{n+1, j} = 0$. Thus, $a, b \in P_{CM}(n+1,t)_0$ and we can define $c \coloneqq  \pi (a)$, and $d \coloneqq \pi (b)$,
	and we have that $c, d \in P_{CM}(n,t)$, $c \neq d$, $x = \frac{1}{2} c + \frac{1}{2} d$, a contradiction.
\end{proof}
\begin{lemma}\label{lemma:vertex-iso2}
	Let $t < n$ and let $y\in \mathbb{R}^{n(n-1)}$ be a vertex of $P_{CM}(n,t)$ of the form
	\begin{equation}
		y_{ij} =
		\begin{cases}
			x_{ij}, & \text{if } i \neq n \neq j,\\
			0,  & \text{else,} 
		\end{cases}
	\end{equation}
	where $x \in \mathbb{R}^{(n-2)(n-1)}$. Then $x$ is a vertex of $P_{CM}(n-1,t)$.
\end{lemma}
\begin{proof} 
	Let
	\begin{align*}
		\begin{aligned}
			i \colon P_{CM}(n,t) &\hookrightarrow P_{CM}(n,t)\\   
			x& \mapsto y
		\end{aligned}
	\end{align*}
	where $x$ and $y$ as in the statement of the theorem. Note that $i$ is well defined because of \Cref{lemma:adding-zeros}, namely, if $y$ is feasible, then $x$ is feasible as well.
	It remains to prove that $x$ is also a vertex. By contradiction, suppose there exist $c, d \in P_{CM}(n-1,t)$ such that $c \neq d$, $x = \frac{1}{2} c + \frac{1}{2} d$. If we define $a \coloneqq i(c)$, and $b \coloneqq i(d)$, we have that $a, b \in P_{CM}(n,t)$, $a \neq b$, $y = \frac{1}{2} a + \frac{1}{2} b$, and so we have a contradiction. 
\end{proof}

Note that the result above still holds if we replace the node $n$ with any node $k \in V \setminus T$.

\begin{observation}\label{obs:small_dim_identific}
	Let $\pi$ and $i$ be the maps introduced in the proofs of \Cref{lemma:vertex-iso1} and \Cref{lemma:vertex-iso2}, respectively. Note that $\pi$ is an injective map and $i(P_{CM}(n,t)) \subset P_{CM}(n+1,t)_0$, thus we have that $\pi$ is also a surjective map and so it is bijective. Moreover, $\pi$ is linear and sends vertices in vertices. In particular $P_{CM}(n+1,t)_0 \cong P_{CM}(n,t)$,
	where the isomorphism is given by the map $\pi$. Note that $\pi$ is a surjective map because given an element $x \in P_{CM}(n,t)$, we have that $\pi (i (x)) = x$, and we can map $i(x)$ through $\pi$ because $i(P_{CM}(n,t)) \subset P_{CM}(n+1,t)_0$.
	This implies that, in the aim of evaluating vertices of our polytopes, it is sufficient to evaluate vertices of $P_{CM}(n,t)$ to get all of the vertices of $P_{CM}(m,t)$, for every $m = t, t+1, \dots, n$. Alternatively, we can evaluate only the spanning vertices of $P_{CM}(n,t)$ for every $n,t$, since every non-spanning vertex can be seen as a spanning vertex of a polytope of a smaller dimension, applying the lemmas above iteratively. Note that we are only interested in non-isomorphic vertices because isomorphic vertices have the same integrality gap. Note also that the results presented above hold for the BCR and the SJ formulations. The proof can be done in almost the same way.
\end{observation}

\begin{observation}\label{obs:adding-ones}
	As we have seen, the trivial way to go from a vertex of $P_{CM}(n+1,t)$ to a vertex of $P_{CM}(n,t)$ is removing zeros, and the trivial way to go from a vertex of $P_{CM}(n,t)$ to a vertex of $P_{CM}(n+1,t)$ is adding zeros. As one would expect, the trivial way to go from a vertex of $P_{CM}(n+1,t+1)$ to a vertex of $P_{CM}(n,t)$ and vice versa is the ``dual'' procedure of the previous one, that is, adding or removing one $1$. Note that this can be done in different ways. More precisely, the following procedures start with a vertex of $P_{CM}(n,t)$ and return a vertex of $P_{CM}(n+1,t+1)$.
	\begin{itemize}
		\item[$(a)$] Add an edge of weight $1$ between a node $v$ of indegree $1$ and the new added terminal, see for example \Cref{fig:oddwheel2} $\to$ \Cref{fig:8-51-2} and \Cref{fig:oddwheel1} $\to$ \Cref{fig:8-51-1}.
		\item[$(b)$] Same as $(a)$, but substituting the outflow of $v$ with the outflow of the newly added terminal, see for example \Cref{fig:oddwheel1} $\to$ \Cref{fig:8-51-3}.
		\item[$(c)$] Add an edge of weight $1$ between the newly added terminal and the root, then swap the role of these two nodes, see \Cref{fig:oddwheel2} $\to$ \Cref{fig:8-51-4}.
	\end{itemize}
	Reversing these procedures, when possible, allows us to go from a $P_{CM}(n+1,t+1)$ to a vertex of $P_{CM}(n,t)$. The proofs are similar to the ones presented above. For all of the procedures above, it is clear that the generated vertices are not isomorphic to the ones we start from.
\end{observation}

Note that the above procedures do not change the integrality gap, as shown by the lemma below, that is actually not restricted to STP.

\begin{lemma}\label{lemma:add-1-IG}
	Let $\epsilon \geq 0$, $c \in \mathbb{R}_{\geq \epsilon}^n$ a cost vector of a minimization ILP instance, and let $x \in \{0,1\}^n$ be the variables of the LP. Denote by $\bar{x}, \hat{x}$ an optimal integer solution and an optimal solution of the LP relaxation, respectively. Suppose an index $k$ exists such that $\hat{x}_k = 1$. Then, the instance $\tilde{c}$ defined as
	\begin{equation}\label{eq:better-cost-IG}
		{\tilde{c}_j} =
		\begin{cases}
			c_j \quad j \neq k, \\
			\epsilon \quad j = k, 
		\end{cases}
	\end{equation}
	has a greater or equal integrality gap than the instance $c$. Moreover, $\hat{x}$ is an optimum for the LP relaxation of the instance $\tilde{c}$.
\end{lemma}

\begin{proof}
	Let $P \subset [0,1]^n$ be the polytope defined by the LP relaxation. We start proving the second claim. Suppose by contradiction that there exists $y \in P$ such that $ \tilde{c}^\intercal y < \tilde{c}^\intercal \hat{x}$. This can be rewritten as
	\begin{equation*}
		\tilde{c}^\intercal y = \sum_{i=1}^n \tilde{c}_i y_i = \epsilon y_k + \sum_{i \neq k}^n c_i y_i < \epsilon \hat{x}_k + \sum_{i \neq k}^n c_i \hat{x}_i.
	\end{equation*}
	This implies
	\begin{equation*}
		\sum_{i \neq k}^n c_i y_i < \epsilon (\hat{x}_k - y_k) + \sum_{i \neq k}^n c_i \hat{x}_i.
	\end{equation*}
	Thus,
	\begin{align*}
		\begin{aligned}
			c^\intercal y &= c_k y_k + \sum_{i \neq k}^n c_i y_i < c_k y_k + \epsilon(\hat{x}_k-y_k) +  \sum_{i \neq k}^n c_i \hat{x}_i = \\
			& = \epsilon(\hat{x}_k-y_k) + c_k y_k - c_k \hat{x}_k + \sum_{i=1}^n c_i \hat{x}_i = \\
			& = (\epsilon - c_k)(\hat{x}_k - y_k) + c^\intercal\hat{x}  
			= (\epsilon - c_k)(1 - y_k) + c^\intercal\hat{x} \leq c^\intercal\hat{x},
		\end{aligned}
	\end{align*}
	which contradicts the optimality of $\hat{x}$ as optimal solution for the LP with cost $c$. For the first part, we have
	\begin{equation*}
		\text{IG}(\tilde{c}) = \frac{\text{ILP}(\tilde{c})}{\text{LP}(\tilde{c})} = \frac{\text{ILP}(\tilde{c})}{\tilde{c}^\intercal \hat{x}} = \frac{\text{ILP}(\tilde{c})}{\text{LP}(c) - c_k + \epsilon}.
	\end{equation*}
	We have that $\text{ILP}(\tilde{c}) \leq \tilde{c}^\intercal \bar{x} = \sum_{i \neq k} c_i x_i + \epsilon x_k \leq \sum_{i \neq k} c_i x_i + c_k x_k = c^\intercal \bar{x} = \text{ILP}(c)$. We now want to prove that $\text{ILP}(\tilde{c}) \geq \text{ILP}(c) -c_k + \epsilon $. Suppose by contradiction that there exists $\bar{y} \in P$ integer such that $\tilde{c}^\intercal \bar{y} < c^\intercal \bar{x} - c_k + \epsilon$. But then
	\begin{align*}
		\begin{aligned}
			c^\intercal \bar{y} &= c_k \bar{y}_k + \sum_{i \neq k}^n c_i \bar{y}_i = (c_k - \epsilon) \bar{y}_k + \tilde{c}^\intercal \bar{y} \\
			&< (c_k - \epsilon) \bar{y}_k + c^\intercal \bar{x} - c_k + \epsilon = (\bar{y}_k-1)(c_k-\epsilon) + c^\intercal \bar{x} \leq c^\intercal \bar{x},
		\end{aligned}
	\end{align*}
	which contradicts the optimality of $\bar{x}$. We then have that $\text{ILP}(c) \geq \text{ILP}(\tilde{c}) \geq \text{ILP}(c) -c_k + \epsilon$ and $\text{LP}(\tilde{c}) = \text{LP}(c) - c_k + \epsilon$, which imply
	\begin{equation*}
		\frac{\text{ILP}(c)}{\text{LP}(c) - c_k + \epsilon}  = \frac{\text{ILP(c)}}{LP(\tilde{c})}\geq \text{IG}(\tilde{c}) \geq \frac{\text{ILP}(c) - c_k + \epsilon}{\text{LP}(c) - c_k + \epsilon} \geq \frac{\text{ILP}(c)}{\text{LP}(c)} = \text{IG}(c)
	\end{equation*}
	where the last inequality holds because $\text{ILP}(c) \geq \text{LP}(c) \geq 0$, $\text{LP}(c) - c_k + \epsilon \geq 0$, $c_k \geq \epsilon$.
\end{proof}

\begin{observation}\label{obs:add-1-IG-stp}
	Let $\hat{x}$ be a fractional vertex of $P_{CM}(n,t)$ such that $\hat{x}_{ij} = 1$. Because of \Cref{lemma:add-1-IG}, the maximum integrality gap is reached by minimizing the value of $c_{ij}$, making node $i$ and node $j$ collapse onto each other. This can be done by choosing a sequence of values of $\epsilon$ such that $\epsilon \to 0$. The study of the gap of this vertex is then equivalent to the study of the gap of a vertex of a smaller dimension. Note that, if we restrict the study to the metric case, even if \eqref{eq:better-cost-IG} might not define a metric cost, the result still holds because of \Cref{lemma:metric-closure-gap}.
\end{observation}
\subsection{Two heuristics procedure for vertices enumeration}
In the following, we state more properties of the CM formulation that permit the design of two different heuristic procedures, in particular, one general search and one dedicated to a specific class of vertices. We are only interested in spanning vertices (\Cref{obs:small_dim_identific}).

\subsubsection{The \texorpdfstring{$\bm{\{1,2\}}$}{1,2}-costs heuristic}
The first procedure is based on the observation that, when looking for integer solutions of the CM formulation, it is enough to study only metric graphs with edge weights in the set $\{1,2\}$.

\begin{theorem}\label{thm:21-metric}
	Let $x \in S_{CM}(n,t)$. Then, $x$ is the unique integer optimal solution for the CM formulation with the metric cost $c_{ij} = 2 - (x_{ij} + x_{ji}) \in \{1, 2\}$.
\end{theorem}

\begin{proof}
	Consider $x$ and the STP instance given by the vector $c$ defined in the statement. We want to prove that $x$ is optimal. Let $x'$ be an integer optimal solution for $c$ and let $s, s'$ be the number of Steiner nodes of $x$ and $x'$, respectively. Let us write $x_e = x_{ij} + x_{ji}$ and the same for $x'$. We will divide the proof into two cases. First, we will prove $(i)$ that if $s' \geq s$, then necessarily $s' = s$ and $x = x'$. Note that this will also give us the uniqueness of the solution. Next, we will prove $(ii)$ that if $s' < s$, it leads to a contradiction.
	\begin{itemize}
		\item[$(i)$] Since the optimal solution is a tree with $t$ terminals and $s$ and $s'$ Steiner node, respectively,  we have that 
		\begin{equation*}
			\sum_{e \in E} x_e = t+s-1, \qquad \sum_{e \in E} x'_e = t+s'-1.
		\end{equation*}
		Note that the definition of $c$ implies that the cost is equal to 1 on the edges of the support graph of the solution and equal to 2 otherwise. 
		Because of the  definition of $c$, we have that
		\begin{equation*}
			\sum_{e \in E} c_ex_e = t+s-1.
		\end{equation*}
		Now let $I_0 = \{e : x_e = 0, \; x_e' = 1\}$, $I_1 = \{e : x_e = 1, \;  x_e' = 0\}$, $I= \{e : x_e =  x_e'\}$. We then have that
		\begin{align*}
			\begin{aligned}
				\sum_{e \in E} c_ex_e' &= \sum_{e \in I_0} c_ex_e' + \sum_{e \in I_1} c_ex_e' + \sum_{e \in I} c_ex_e' = 2 \times |I_0| + |I| \geq |I_0| + |I|= \\
				& = t + s' - 1 \geq t+s-1 = \sum_{e \in E} c_ex_e,
			\end{aligned}
		\end{align*}
		and this, together with the fact that $x'$ is optimal for $c$, implies that all the inequalities are actually equalities. This is realized if and only if $s' = s$ and $|I_0| = 0$. Since $s' = s$,  we have that $x$ and $x'$ have the same number of edges, and given the fact that $|I_0| = 0$, we cannot have that $|I_1| \geq 1$, and so $|I_1| = 0$. By the fact that, fixed the root, there exists a unique orientation, we have that $x$ and $x'$ coincide.
		\item[$(ii)$] Let $\mathcal{S}(x)$ and $\mathcal{S}(x')$ be the set of Steiner nodes of the solution $x$ and $x'$, respectively.
		Let $\mathcal{S}= \{s_1, \dots, s_k\} = \mathcal{S}(x) \setminus \mathcal{S}(x')$ and let $z$ be the number of edges of $x$ with extreme points $s_i, s_j \in \mathcal{S}(x)$. Note that $\mathcal{S} \neq \emptyset$, otherwise we would have $\mathcal{S}(x) \subset \mathcal{S}(x')$ and so $s' \geq s$. Now, $x'$ is a tree with $s'+t$ nodes. Thanks to the hypothesis $s' < s$, we have  $s'= s - k$, and hence $x'$ is a tree with $s+t-k$ nodes, so $s+t-k-1$ edges. On the other side, $x$ has $s+t-1$ edges, all of them of cost $1$, while all of the other edges have cost $2$. We now evaluate how many edges of cost $2$ $x'$ must have, given that it does not contain any node of the set $\mathcal{S}$. We have that $c$ contains exactly $s+t-1$ edges of cost $1$, and the number of those edges that contain a node of $\mathcal{S}$ is $\left(\sum_{i=1}^k deg(s_i)\right) - z$. 
		
		Since we said that $x'$ must contain $s+t-k-1$ edges, its cost is $E_1 + 2 \times E_2$, where $E_1$ is the number of edges of cost $1$ and $E_2$ is the number of edges of cost 2, and we have that
		\begin{equation*}
			E_1 \leq  s+t-1-\left(\left(\sum_{i=1}^k deg(s_i)\right) - z\right), \quad 
			E_2  = s+t-k-1 - E_1,
		\end{equation*}
		and the minimum of $E_1 + 2 \times E_2 = 2(s+t-k-1) - E_1$ is attained when $E_1$ is exactly equal to its upper bound. We then have that
		\begin{align*}
			\begin{aligned}
				c^\intercal x' - c^\intercal x &= E_1 + 2 \times E_2 - (s+t-1) = s +t - 2k -1 - E_1  = \\
				& \geq s+t-2k-1-s-t-1+\left(\left(\sum_{i=1}^k deg(s_i)\right) - z\right) = \\
				& = \left(\sum_{i=1}^k deg(s_i)\right) - 2k -z \geq 3k - 2k -z \geq k -z \geq 1,
			\end{aligned}
		\end{align*}
		So, $x'$ is not optimal, and we have a contradiction. Note that $deg(s_i) \geq 3$ because of Constraints \eqref{eq:cm_inflow} and \eqref{eq:cm_in_out_flow}, and $k -z \geq 1$ because the support graph associated to $\mathcal{S}$ as a subgraph of $x$ is a forest since it is a subgraph of a CM solution, which is a tree.
	\end{itemize}
\end{proof}

\begin{observation}\label{obs:1-2-not-exha}
	Note that the generalization of \Cref{thm:21-metric} does not hold in general for the non-integer case, that is, if $x$ is a non-integer point of $P_{CM}(n,t)$, then $x$ is not necessarily an optimal solution for the CM formulation with the metric cost
	\begin{equation}\label{eq:2-1-cost-fract}
		c_{ij} = 2 - \mathds{1}_{E_x}(\{i,j\}), \qquad E_x = \{e = \{i,j\} \in E  \, : \, x_{ij} + x_{ji} > 0\},
	\end{equation}
	see, for example, the vertex depicted in \Cref{fig:8-5-2}. In this case, with the cost assignation \eqref{eq:2-1-cost-fract}, we have that the fractional vertex has a value of $11/2$ (multiply the number of edges by 1/2), while the optimal value of the CM formulation for this instance is $5$. Thus, the vertex shown in \Cref{fig:8-5-2} cannot be optimal for this instance. It still holds that the vertex mentioned above is an optimum of a metric graph where every edge weight is in the set $\{1, 2\}$, namely setting the cost as in \eqref{eq:2-1-cost-fract} but changing the cost of the two edges outflowing the root, setting them to $2$ instead of $1$. Note that in this case, the subgraph linked to edges with cost $1$ is not connected, as the root represents a connected component.
\end{observation}
The observation above, together with \Cref{thm:21-metric}, lead us to formulate a heuristic search based on the generation of metric graphs with edge weights in the set $\{1, 2\}$ and then solve the STP on those instances. The detailed procedure called $\text{OTC}(n,t)$ as in One-Two-Costs is described in \Cref{alg:1-2}.

Note that for computational reasons, we restricted our search to the generation of only connected graphs, and so to graphs with costs $\{1,2\}$ in which the subgraph regarding the edges of cost $1$ spans all the nodes and is connected. We know this is a strong restriction, making the procedure unable to find some vertices, see \Cref{obs:1-2-not-exha}. Note also that we restrict our search to graphs $G = (V, E)$ with $ n \leq |E| \leq n\cdot t - t^2$: the lower bound is given by the fact that we are only interested in non-integer vertices, and the upper bound was derived after a first set of computational experiments.

\begin{algorithm}[t!]
	\caption{$\{1,2\}$-costs vertices heuristic}
	\label{alg:1-2}
	\begin{algorithmic}
		\STATE{\textbf{procedure} OTC$(n,t)$}
		\STATE{$\mathbb{G} = \{G = (V,E)\; | \; G \text{ connected}, \; |V| = n, \; n \leq |E| \leq n \cdot t - t^2\}$}
		\STATE{$\mathbb{T} = \{T \; | \;  T \subset \{1, \dots, n\}, \; |T| = t\}$}
		\STATE{$\mathfrak{G} = \emptyset$}
		\FOR{$G \in \mathbb{G}$}
		\FOR{$T \in \mathbb{T}$}
		\FOR{$r \in T$}
		\STATE{$G_{T,r} = $ node-colored graph with}
		\STATE{\hspace*{9.5em}$\cdot$ $G$ as its support graph}
		\STATE{\hspace*{9.5em}$\cdot$ $r$ colored as \texttt{root}}
		\STATE{\hspace*{9.5em}$\cdot$ $i$ colored as \texttt{terminal} $\forall i \in T \setminus \{r\}$}
		\STATE{\hspace*{9.5em}$\cdot$ $j$ colored as \texttt{steiner} $\forall j \notin T$}
		\IF{$ G_{T, r} \ncong H  \; \forall H \in \mathfrak{G}$}
		\STATE{add $G_{T, r}$ to $\mathfrak{G}$}
		\ENDIF
		\ENDFOR
		\ENDFOR
		\ENDFOR
		\STATE{$\mathcal{V} = \emptyset$}
		\FOR{$G_{T,r} \in \mathfrak{G}$}
		\STATE{obtain the STP instance $(G, T, r)$ from $G_{T, r}$ with $c_{ij} = 
			\begin{cases}
				1 \quad \text{if } \{i,j\} \in G_{T, r}\\
				2 \quad \text{otherwise}
			\end{cases}$}
		\STATE{solve the relaxation of the CM formulation for the instance above}
		\IF{a solution $x$ is found and it is a non-integer vertex of $P_{CM}(n,t)$}
		\STATE{add $x$ to $\mathcal{V}$}
		\ENDIF
		\ENDFOR
		\RETURN$\mathcal{V}$
	\end{algorithmic}
\end{algorithm}

\subsubsection{Pure half-integer vertices}
Guided by the literature, we narrow our research to a specific class of vertices, which is conjectured to exhibit the maximum integrality gap in other NP-Hard problems (see, e.g., the Symmetric Traveling Salesman Problem~\cite{benoit2008finding, boyd2007structure} and the Asymmetric Traveling Salesman Problem~\cite{elliott2008integrality}). In particular, we restrict our attention to vertices having their values in $\{0, \frac{1}{2}, 1\}$.
Given a non-integer vertex $x$ of $P_{CM}(n,t)$, we say that $x$ is \textit{half-integer} (HI) if $x_{ij} \in \{0, 1/2, 1\} \; \forall i, j \in V$ and we say that $x$ is \textit{pure half-integer} (PHI) if $x_{ij} \in \{0, 1/2\} \; \forall i, j \in V$. 
In the following section, we state and prove some properties of PHI vertices. The choice of focusing only on PHI vertices instead of HI vertices is motivated by \Cref{lemma:add-1-IG} and \Cref{obs:add-1-IG-stp}.

\begin{lemma}\label{lemma:oriented-cycles}
	Let $x$ be a pure half-integer solution of $P_{CM}(n,t)$ which is also a vertex of $P_{BCR}(n,t)$ that is optimum for a metric cost. Then we have that $x_{ij} > 0 \Rightarrow x_{ji}~=~0$. 
\end{lemma}

\begin{proof}
	Since $x$ is pure half-integer, we have $x_{ij} = 1/2$. Suppose by contradiction that $x_{ji} \neq 0$, and so by the same reasoning, $x_{ji} = 1/2$. Because of \Cref{lemma:cm_connected}, we have that the set $\{i,j\}$ is not a connected component of $x$, namely, is not an isolated 2-cycle, and neither of the two nodes can be the root, as the root has inflow equal to $0$ because of Constraint \eqref{eq:cm_root_inflow}. Thus, there must exist a path from the root to the two nodes, and so there must exist an active arc going from a third node to one of the two nodes we are considering. Without loss of generality, let $x_{ki} > 0$, that implies $x_{ki} = 1/2$. Suppose $x_{ik} = 0 $ since, otherwise, we can do the same reasoning for nodes $\{i, j, k\}$ and repeat it until we return to the root, which has no inflow. Now, we distinguish between two cases. 
	
	\begin{enumerate}[(a)]
		\item  No other inflow is present in $j$, that is, $x_{aj} = 0 \; \forall a \neq j$ . Note that this implies that $j$ is not a terminal since it has an inflow of $1/2$. Then $x$ is not optimum. Consider $x'$ that is equal to $x$ on all the arcs but the arc $(j, i)$, and set $x'_{ji} = 0$. For any nonnegative $c$, $c^\intercal x' < c^\intercal x$. Note that $x'$ is feasible for the BCR. Constraint \eqref{eq:dcut_degree} is clearly satisfied. Constraint \eqref{eq:dcut_subt} could not be verified by $x'$ only for a set $W$ for which $ i \in W, \; j \notin W$, because then the only variables that differs from $x$ appears. Let us take one of these sets and define $\bar{W} = W \cup \{j\}$. We can write
		\begin{align*}
			\sum_{(a,b) \in \delta^-(W)} x_{ab}' &= \sum_{\substack{(a,b) \in \delta^-(W)\\ (a,b) \neq (j,i)}} x_{ab}' + x_{ji}' = \sum_{(a,b) \in \delta^-(\bar{W})} x_{ab}' - \sum_{a \in V \setminus W} x_{aj}' + x_{ji}' = \\
			&= \sum_{(a,b) \in \delta^-(\bar{W})} x_{ab} - \sum_{a \in V \setminus W} x_{aj} + x_{ji}' = \\
			& =\sum_{(a,b) \in \delta^-(\bar{W})} x_{ab} +0 + x_{ji}' \geq 1 + 0 + 0 = 1,
		\end{align*}
		where the inequality holds because $x$ is feasible and $W$ is a valid set. So we have that $x'$ is feasible even for the constraints regarding the sets  $W$ for which $i \in W, \; j \notin W$ and so it is feasible for the BCR.
		If $x'$ is feasible for the CM, the proof is concluded. 
		If $x'$ is not feasible for the CM formulation, it is because of Constraint \eqref{eq:cm_in_out_flow} because $x'$ satisfied all of the other constraints since $x$ is feasible for CM.
		Regarding Constraint \eqref{eq:cm_in_out_flow}, if $x'$ is not feasible for the CM anymore, it is because the outdegree of $j$ in $x$ was exactly two, namely $x_{ji} = \frac{1}{2}$ and there exists $d$ such that $x_{jd} = \frac{1}{2}$.
		Hence, we can build $x''$ from $x'$ by removing arc $x'_{ij}$ and $x'_{jd}$ from $x'$ and by adding the arc $x''_{id}$, avoiding the detour in $j$. 
		This solution is feasible for the CM, and it holds
		$c^\intercal x'' \leq c^\intercal x'$,
		for the non-negativity and the triangle inequality. 
		Hence,
		$c^ \intercal x'' < c^\intercal x$, 
		from the relation between $x$ and $x'$ already proved.
		Hence, we can conclude that if the only inflow of the (Steiner) node $j$ is $x_{ij}$, $x$ is neither optimal for the CM nor for the BCR.
		\item The total inflow of $j$ is $1$, and so there exists $l$ such that $x_{lj} = 1/2$. Suppose $x_{jl} = 0$ and suppose that both $k$ and $l$ have an inflow of $1$. This will ensure the feasibility of the two points we are about to construct. Then $x$ is not a vertex of $P_{BCR}(n,t)$, because by setting 
		\begin{align*}
			y_{ab} & =
			\begin{cases}
				0, & \text{if } a = l, \; b = j, \; \text{or} \; a = j, \; b = i,\\
				1, & \text{if } a = i, \; b = j, \; \text{or} \;  a = k, \; b = i,\\
				x_{ab},  & \text{else,} 
			\end{cases} \\
			z_{ab} & =
			\begin{cases}
				1, & \text{if } a = l, \; b = j, \; \text{or} \; a = j, \; b = i,\\
				0, & \text{if } a = i, \; b = j, \; \text{or} \;  a = k, \; b = i,\\
				x_{ab},  & \text{else,} 
			\end{cases}
		\end{align*}
		we have $y \neq z$, $x = \frac{1}{2} y + \frac{1}{2} z$, and $y,z \in P_{BCR}(n,t)$ by an argument similar from the one above.
		
		Visually, we can represent the three points as the following 
		
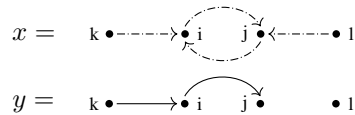
{\begin{center}
				\begin{tikzpicture}
					\node[] at (-1,0) (0) {$x =$};
					\fill[] (0,0) circle (0.05);
					\fill[] (1,0) circle (0.05);
					\fill[] (2,0) circle (0.05);
					\fill[] (3,0) circle (0.05);
					\node[] at (-0.2,0) (0) {\scriptsize k};
					\node[] at (1.2,0) (1) {\scriptsize i};
					\node[] at (1.8,0) (2) {\scriptsize j};
					\node[] at (3.2,0) (3) {\scriptsize l};
					\draw[->, densely dashdotted] (0.1,0) -- (0.9,0);
					\draw[->, densely dashdotted] (1,0.1) to [out=60,in=120] (2,0.1);
					\draw[->, densely dashdotted] (2,-0.1) to [out=240,in=300] (1,-0.1);
					\draw[->, densely dashdotted] (2.9,0) -- (2.1,0);
				\end{tikzpicture}
				\\
				\begin{tikzpicture}
					\node[] at (-1,0) (0) {$y =$};
					\fill[] (0,0) circle (0.05);
					\fill[] (1,0) circle (0.05);
					\fill[] (2,0) circle (0.05);
					\fill[] (3,0) circle (0.05);
					\node[] at (-0.2,0) (0) {\scriptsize k};
					\node[] at (1.2,0) (1) {\scriptsize i};
					\node[] at (1.8,0) (2) {\scriptsize j};
					\node[] at (3.2,0) (3) {\scriptsize l};
					\draw[->] (0.1,0) -- (0.9,0);
					\draw[->] (1,0.1) to [out=60,in=120] (2,0.1);
				\end{tikzpicture}
				\\[10pt]
				\begin{tikzpicture}
					\node[] at (-1,0) (0) {$z =$};
					\fill[] (0,0) circle (0.05);
					\fill[] (1,0) circle (0.05);
					\fill[] (2,0) circle (0.05);
					\fill[] (3,0) circle (0.05);
					\node[] at (-0.2,0) (0) {\scriptsize k};
					\node[] at (1.2,0) (1) {\scriptsize i};
					\node[] at (1.8,0) (2) {\scriptsize j};
					\node[] at (3.2,0) (3) {\scriptsize l};
					\draw[->] (2,-0.1) to [out=240,in=300] (1,-0.1);
					\draw[->] (2.9,0) -- (2.1,0);
				\end{tikzpicture}
		\end{center}}
		where we draw only the interesting arcs. Note that dashed arcs represent a value of $1/2$ while full arcs represent a value of $1$.
		If $x_{jl} \neq 0$, that is, $x_{jl} = 1/2$, then we can go backward until we find one node $m$ such that there exists $p$ for which $x_{pm} = 1/2,\;  x_{mp} = 0$, and such a $p$ exists because we can go back to the root with the same reasoning as above. Suppose that both $k$ and $p$ have an inflow of $1$. We now do the same reasoning with $y$ and $z$ but consider the whole paths from $p$ to $i$ and from $k$ to $m$ instead of the paths from $l$ to $i$ and from $k$ to $j$.
		
		Visually, we can represent the three points as the following 
		
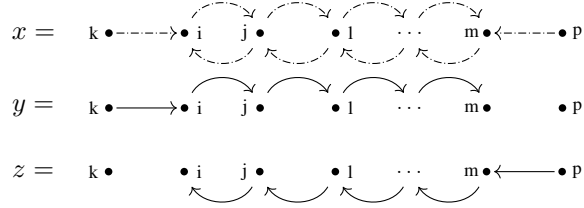
{\begin{center}
				\begin{tikzpicture}
					\node[] at (-1,0) (0) {$x =$};
					\fill[] (0,0) circle (0.05);
					\fill[] (1,0) circle (0.05);
					\fill[] (2,0) circle (0.05);
					\fill[] (3,0) circle (0.05);
					\node[] at (-0.2,0) (0) {\scriptsize k};
					\node[] at (1.2,0) (1) {\scriptsize i};
					\node[] at (1.8,0) (2) {\scriptsize j};
					\node[] at (3.2,0) (3) {\scriptsize l};
					\draw[->, densely dashdotted] (0.1,0) -- (0.9,0);
					\draw[->, densely dashdotted] (1.1,0.2) to [out=60,in=120] (1.9,0.2);
					\draw[->, densely dashdotted] (1.9,-0.2) to [out=240,in=300] (1.1,-0.2);
					\draw[->, densely dashdotted] (2.9,-0.2) to [out=240,in=300] (2.1,-0.2);
					\draw[->, densely dashdotted] (2.1,0.2) to [out=60,in=120] (2.9,0.2);
					\draw[->, densely dashdotted] (3.9,-0.2) to [out=240,in=300] (3.1,-0.2);
					\draw[->, densely dashdotted] (3.1,0.2) to [out=60,in=120] (3.9,0.2);
					\node[] at (4,0) (3) {\scriptsize $\dots$};
					\draw[->, densely dashdotted] (4.9,-0.2) to [out=240,in=300] (4.1,-0.2);
					\draw[->, densely dashdotted] (4.1,0.2) to [out=60,in=120] (4.9,0.2);
					\fill[] (5,0) circle (0.05);
					\node[] at (4.8,0) (3) {\scriptsize m};
					\fill[] (6,0) circle (0.05);
					\node[] at (6.2,0) (3) {\scriptsize p};
					\draw[->, densely dashdotted] (5.9,0) -- (5.1,0);        
				\end{tikzpicture}
				\\
				\begin{tikzpicture}
					\node[] at (-1,0) (0) {$y =$};
					\fill[] (0,0) circle (0.05);
					\fill[] (1,0) circle (0.05);
					\fill[] (2,0) circle (0.05);
					\fill[] (3,0) circle (0.05);
					\node[] at (-0.2,0) (0) {\scriptsize k};
					\node[] at (1.2,0) (1) {\scriptsize i};
					\node[] at (1.8,0) (2) {\scriptsize j};
					\node[] at (3.2,0) (3) {\scriptsize l};
					\draw[->] (0.1,0) -- (0.9,0);
					\draw[->] (1.1,0.2) to [out=60,in=120] (1.9,0.2);
					\draw[->] (2.1,0.2) to [out=60,in=120] (2.9,0.2);
					\draw[->] (3.1,0.2) to [out=60,in=120] (3.9,0.2);
					\node[] at (4,0) (3) {\scriptsize $\dots$};
					\draw[->] (4.1,0.2) to [out=60,in=120] (4.9,0.2);
					\fill[] (5,0) circle (0.05);
					\node[] at (4.8,0) (3) {\scriptsize m};
					\fill[] (6,0) circle (0.05);
					\node[] at (6.2,0) (3) {\scriptsize p};
				\end{tikzpicture}
				\\[10pt]
				\begin{tikzpicture}
					\node[] at (-1,0) (0) {$z =$};
					\fill[] (0,0) circle (0.05);
					\fill[] (1,0) circle (0.05);
					\fill[] (2,0) circle (0.05);
					\fill[] (3,0) circle (0.05);
					\node[] at (-0.2,0) (0) {\scriptsize k};
					\node[] at (1.2,0) (1) {\scriptsize i};
					\node[] at (1.8,0) (2) {\scriptsize j};
					\node[] at (3.2,0) (3) {\scriptsize l};
					\draw[->] (1.9,-0.2) to [out=240,in=300] (1.1,-0.2);
					\draw[->] (2.9,-0.2) to [out=240,in=300] (2.1,-0.2);
					\draw[->] (3.9,-0.2) to [out=240,in=300] (3.1,-0.2);
					\node[] at (4,0) (3) {\scriptsize $\dots$};
					\draw[->] (4.9,-0.2) to [out=240,in=300] (4.1,-0.2);
					\fill[] (5,0) circle (0.05);
					\node[] at (4.8,0) (3) {\scriptsize m};
					\fill[] (6,0) circle (0.05);
					\node[] at (6.2,0) (3) {\scriptsize p};
					\draw[->] (5.9,0) -- (5.1,0);
				\end{tikzpicture}
		\end{center}}
		where we draw only the interesting arcs. Note that dashed arcs represent a value of $1/2$ while full arcs represent a value of $1$. If $k$ or $p$ do not have an inflow of $1$, we can just go backward until we find a point with this property. If we do not find it, we go backward to the root. At this point, we can do the same reasoning with the paths as we did above. 
	\end{enumerate}
\end{proof}
We now focus on a particular type of PHI vertices, namely spanning vertices such that every Steiner node has indegree exactly one. We conjecture that every PHI spanning vertex has this property. Further details can be found in \Cref{app:PHI-sp}.

We now derive some properties of these vertices that will be exploited in our heuristic search.
\begin{lemma}\label{lemma:num-edges}
	Let $x$ be a pure half-integer solution of $P_{CM}(n,t)$, $t \geq 3$ that is also a vertex of $P_{BCR}(n,t)$ and an optimum for a metric cost. Let $x$ be a spanning vertex such that every Steiner node has indegree $1$. Then, it holds that 
	\begin{itemize}
		\item[(1)] $|\{(i, j) \in A \; | \; x_{ij} > 0\}| = n+t-2$,
		\item[(2)] $3t-n-4 \geq 0$.
	\end{itemize}
\end{lemma}
\begin{proof}
	For the first point, it suffices to count the incoming edges of each node. We have one incoming edge for each Steiner node and exactly two incoming edges for every terminal that is not the root since every terminal has an inflow exactly equal to one, and our edges have weights $1/2$. The total number of edges is $n-t + 2 (t- 1) = n+t-2$.
	
	For the second point, because of Constraint \eqref{eq:cm_subt}, we have that at least two edges exit from the root and at least two edges enter in every other terminal. Moreover, since in every Steiner node enters exactly one edge, at least two edges must come out. By counting all the edges mentioned in point 1 of this proposition, and considering the minimum number just discussed we then have that $(n+t-2) \geq \frac{2t + 3 (n-t)}{2}$ and so $3t-n-4 \geq 0$.
\end{proof}
The properties stated above represent the core of the heuristic we now present. We generate all of the non-isomorphic connected undirected graphs such that every node is of degree at least $2$ and with exactly $n+t-2$ edges with the command \texttt{geng} of nauty~\cite{mckay2014practical}. For every generated graph, we generate all the non-isomorphic orientation of the edges, that can only be oriented in one way because of \Cref{lemma:oriented-cycles}, and such that every node has a maximum indegree of $2$ since we have Constraint \eqref{eq:cm_inflow} and we are dealing with PHI solutions. This generation of digraphs can be done with the command \texttt{watercluster2} of nauty. The obtained digraphs can be mapped into a spanning PHI vertex of $P_{CM}(n,t)$ for every feasible case. In particular, we have to check that: (i) There exist exactly $n - t$ nodes with in-degree $1$ (Steiner nodes); (ii) There exists one node with in-degree $0$ (root); (iii) There exist exactly $t-1$ nodes of in-degree $2$ (terminals). We filter all the generated graphs for these properties and then check if the remaining ones are vertices of $P_{CM}(n,t)$. This procedure called $\text{PHI}(n,t)$, is illustrated in \Cref{alg:phi}.

\begin{observation}
	Note how the $\text{PHI}(n,t)$ can be generalized to vertex attaining values in the set $\{0, 1/m\}$ just by changing some values: the indegree of the terminal nodes must now be $m$, as well as the outdegree of the root, while the indegree of the Steiner nodes is again $1$. This gives us a total number of edges of $n+(m-1)\times t - m$. In addition, every node has degree at least $\min(3, m)$; if $m > 3$ the number of nodes with degree $3$ is at most $n-t$; there must exist one node of indegree $0$, $n-t$ nodes of indegree $1$, and $t-1$ nodes of indegree $m$. In \Cref{subsec:beyondphi}, the case of $m=4$ is discussed in more detail, with the algorithm for this particular case being presented in \Cref{app:POQ}.
\end{observation}
\begin{algorithm}[t!]
	\caption{Pure half-integer vertices search}\label{alg:phi}
	\begin{algorithmic}
		\STATE{\textbf{procedure} PHI$(n,t)$}
		\STATE{$\mathbb{G} = \{G =(V,E) \; | \; G \text{ connected}, \; deg(i) \geq 2 \, \forall i \in V, \; |V| = n, \; |E| = n+t-2\}$}
		\STATE{$\text{di}\mathbb{G} = \emptyset$}
		\FOR{$G=(V,E) \in \mathbb{G}$}
		\IF{$|\{i \in V \; | \; deg(i) = 2 \}| \leq t$}
		\STATE{add to $\text{di}\mathbb{G}$ every non-isomorphic orientation of $G$ s.t.}
		\STATE{\hspace*{4.5em}$\cdot$ every edge can be oriented in only one way}
		\STATE{\hspace*{4.5em}$\cdot$ every node has a maximum indegree of $2$}
		\ENDIF
		\ENDFOR
		\STATE{$\mathcal{V} = \emptyset$}
		\FOR{$\text{di}G = (V, A) \in \text{di}\mathbb{G}$}
		\IF{$|\{i \in V \; | \; indeg(i) = 0 \}| = 1$}
		\IF{$|\{i \in V \; | \; indeg(i) = 1 \}| = n-t$}
		\IF{$|\{i \in V \; | \; indeg(i) = 2 \}| = t-1$}
		\STATE{$x_{ij} = 1/2 \text{ iff } (i,j) \in A$ is a solution of $P_{CM}(n,t)$ with}
		\STATE{\hspace*{7.5em}$\cdot$ $\{r\} = \{i \in V \; | \; indeg(i) = 0 \}$}
		\STATE{\hspace*{7.5em}$\cdot$ $V \setminus T = \{i \in V \; | \; indeg(i) = 1 \}$}
		\STATE{\hspace*{7.5em}$\cdot$ $T \setminus \{r\} = \{i \in V \; | \; indeg(i) = 2 \}$}
		\IF{$x$ is a feasible vertex of $P_{CM}(n,t)$}
		\STATE{add $x$ to $\mathcal{V}$}
		\ENDIF
		\ENDIF
		\ENDIF
		\ENDIF
		\ENDFOR
		\RETURN $\mathcal{V}$
	\end{algorithmic}
\end{algorithm}

\section{Computational results}\label{sec:results}
In this section, we aim to generate vertices of the $P_{CM}$ polytope and, for any vertex, evaluate the maximum integrality gap that can be attained at that vertex. Recall that vertices of the $P_{BCR}$ polytope that are feasible for the CM formulation are also vertices of the $P_{CM}$ polytope.

Specifically, our approach involves: (i) running the two proposed heuristics for small values of \( n \) to produce vertices of $P_{CM}$, (ii) computing the maximum integrality gap associated with these vertices by solving the Gap problem, and (iii) extending our analysis beyond PHI vertices.
\paragraph{Implementation details} All the tests are executed on a desktop computer with a CPU 13th Gen Intel(R) Core(TM) i5-13600 and 16 GB of RAM. All the functions have been implemented in Python. 
For the optimization tasks, we use the commercial solver Gurobi 11.0~\cite{gurobi}.

\subsection{Lower bounds for the integrality gap for \texorpdfstring{$\bm{n \leq 10}$}{n <= 10}}

Our first set of algorithmic experiments aims at generating nontrivial vertices of $P_{CM}$ having a large integrality gap.
\Cref{tab:OTC_vs_PHI} and \Cref{tab:PHI>8} present the lower bounds on the integrality gap and the number of nonintegral vertices we can compute with our two heuristics presented in \Cref{sec:vertex_enum}.
Recall that the vertices computed by the OTC procedures are filtered by isomorphism in post-processing. 
In the following paragraphs, we discuss the results for specific values of the number of vertices, starting from $n=6$. Recall that results for $n \leq 5$ are presented in \Cref{tab:vertices-dcut-cm}: we computed every vertex with Polymake, finding a maximum value of integrality gap equal to $1$. Further details can be found in \Cref{app:polymake}.

\paragraph{$(n = 6)$} For $n = 6$, for any value of $t \leq 3 \leq n - 1$, the best lower bound we compute is always equal to 1. We conjecture that for $n = 6$, the CM and the BCR formulation have a gap equal to 1. 

\paragraph{$(n = 7)$} For $ n = 7$, we compute four vertices attaining the gap of $\frac{10}{9}$ with the heuristic OTC; three of them belong to the same class of isomorphism. \Cref{fig:7-4} shows those four vertices, where \Cref{fig:oddwheel1,fig:oddwheel3,fig:oddwheel4} show the three isomorphic graphs.
Moreover, these vertices are PHI.
Although the directed support graph is the same for the four vertices, the arc orientation changes, and, in particular, the node labeled as ``root'' is different.
Note that the PHI heuristic can only find two of the four vertices since it can find only one vertex for every class of isomorphism of node-colored edge-weighted directed graphs.
On the contrary, the OTC heuristic may find more than one representative for the same class of isomorphism. 

\begin{figure}[t!]
	\centering
	\subfloat[t!][First vertex. \label{fig:oddwheel2}]{
		\scalebox{0.5}{\begin{tikzpicture}
				\draw[line width=3, color=mycol4] (0:0) -- (90:2.5) node [currarrow, pos=0.5, sloped] {};
				\draw[line width=3, color=mycol4] (0:0) -- (210:2.5) node [currarrow, pos=0.5, sloped, rotate=180] {};
				\draw[line width=3, color=mycol4] (0:0) -- (330:2.5) node [currarrow, pos=0.5, sloped] {};
				\draw[line width=3, color=mycol4] (30:2.5) -- (90:2.5) node [currarrow, pos=0.5, sloped] {};
				\draw[line width=3, color=mycol4] (90:2.5) -- (150:2.5) node [currarrow, pos=0.5, sloped, rotate=180] {};
				\draw[line width=3, color=mycol4] (150:2.5) -- (210:2.5) node [currarrow, pos=0.5, sloped, rotate=180] {};
				\draw[line width=3, color=mycol4] (210:2.5) -- (270:2.5) node [currarrow, pos=0.5, sloped] {};
				\draw[line width=3, color=mycol4] (270:2.5) -- (330:2.5) node [currarrow, pos=0.5, sloped, rotate=180] {};
				\draw[line width=3, color=mycol4] (330:2.5) -- (30:2.5) node [currarrow, pos=0.5, sloped] {};
				\fill[color=mycol1] (0:0) circle (0.2);
				\fill[color=white] (0:0) circle (0.08);
				\fill[color=mycol3] (30:2.5) circle (0.2);
				\fill[color=mycol3] (150:2.5) circle (0.2);
				\fill[color=mycol3] (270:2.5) circle (0.2);
				\fill[rotate around={90:(0,0)}, color=mycol2] (2.3, -0.2) rectangle (2.7, 0.2);
				\fill[rotate around={210:(0,0)}, color=mycol2] (2.3, -0.2) rectangle (2.7, 0.2);
				\fill[rotate around={330:(0,0)}, color=mycol2] (2.3, -0.2) rectangle (2.7, 0.2);
		\end{tikzpicture}}
	}
	\hspace*{0.1cm}
	\subfloat[t!][Second vertex.\label{fig:oddwheel1}]{
		\scalebox{0.5}{\begin{tikzpicture}
				\draw[line width=3, color=mycol4] (0:0) -- (90:2.5) node [currarrow, pos=0.5, sloped, rotate=180] {};
				\draw[line width=3, color=mycol4] (0:0) -- (210:2.5) node [currarrow, pos=0.5, sloped] {};
				\draw[line width=3, color=mycol4] (0:0) -- (330:2.5) node [currarrow, pos=0.5, sloped] {};
				\draw[line width=3, color=mycol4] (30:2.5) -- (90:2.5) node [currarrow, pos=0.5, sloped] {};
				\draw[line width=3, color=mycol4] (90:2.5) -- (150:2.5) node [currarrow, pos=0.5, sloped] {};
				\draw[line width=3, color=mycol4] (150:2.5) -- (210:2.5) node [currarrow, pos=0.5, sloped] {};
				\draw[line width=3, color=mycol4] (210:2.5) -- (270:2.5) node [currarrow, pos=0.5, sloped] {};
				\draw[line width=3, color=mycol4] (270:2.5) -- (330:2.5) node [currarrow, pos=0.5, sloped, rotate=180] {};
				\draw[line width=3, color=mycol4] (330:2.5) -- (30:2.5) node [currarrow, pos=0.5, sloped] {};
				\fill[color=mycol3] (0:0) circle (0.2);
				\fill[color=mycol3] (30:2.5) circle (0.2);
				\fill[color=mycol1] (150:2.5) circle (0.2);
				\fill[color=white] (150:2.5) circle (0.08);
				\fill[color=mycol3] (270:2.5) circle (0.2);
				\fill[rotate around={90:(0,0)}, color=mycol2] (2.3, -0.2) rectangle (2.7, 0.2);
				\fill[rotate around={210:(0,0)}, color=mycol2] (2.3, -0.2) rectangle (2.7, 0.2);
				\fill[rotate around={330:(0,0)}, color=mycol2] (2.3, -0.2) rectangle (2.7, 0.2);
	\end{tikzpicture}}} \hspace*{0.1cm}
	\subfloat[t!][Third vertex.\label{fig:oddwheel3}]{
		\scalebox{0.5}{\begin{tikzpicture}
				\draw[line width=3, color=mycol4] (0:0) -- (90:2.5) node [currarrow, pos=0.5, sloped, rotate=180] {};
				\draw[line width=3, color=mycol4] (0:0) -- (210:2.5) node [currarrow, pos=0.5, sloped, rotate=180] {};
				\draw[line width=3, color=mycol4] (0:0) -- (330:2.5) node [currarrow, pos=0.5, sloped, rotate=180] {};
				\draw[line width=3, color=mycol4] (30:2.5) -- (90:2.5) node [currarrow, pos=0.5, sloped, rotate=180] {};
				\draw[line width=3, color=mycol4] (90:2.5) -- (150:2.5) node [currarrow, pos=0.5, sloped, rotate=180] {};
				\draw[line width=3, color=mycol4] (150:2.5) -- (210:2.5) node [currarrow, pos=0.5, sloped, rotate=180] {};
				\draw[line width=3, color=mycol4] (210:2.5) -- (270:2.5) node [currarrow, pos=0.5, sloped] {};
				\draw[line width=3, color=mycol4] (270:2.5) -- (330:2.5) node [currarrow, pos=0.5, sloped, rotate=180] {};
				\draw[line width=3, color=mycol4] (330:2.5) -- (30:2.5) node [currarrow, pos=0.5, sloped, rotate=180] {};
				\fill[color=mycol3] (0:0) circle (0.2);
				\fill[color=mycol1] (30:2.5) circle (0.2);
				\fill[color=white] (30:2.5) circle (0.08);
				\fill[color=mycol3] (150:2.5) circle (0.2);
				\fill[color=mycol3] (270:2.5) circle (0.2);
				\fill[rotate around={90:(0,0)}, color=mycol2] (2.3, -0.2) rectangle (2.7, 0.2);
				\fill[rotate around={210:(0,0)}, color=mycol2] (2.3, -0.2) rectangle (2.7, 0.2);
				\fill[rotate around={330:(0,0)}, color=mycol2] (2.3, -0.2) rectangle (2.7, 0.2);
	\end{tikzpicture}}}\hspace*{0.1cm}
	\subfloat[t!][Fourth vertex.\label{fig:oddwheel4}]{
		\scalebox{0.5}{\begin{tikzpicture}
				\draw[line width=3, color=mycol4] (0:0) -- (90:2.5) node [currarrow, pos=0.5, sloped] {};
				\draw[line width=3, color=mycol4] (0:0) -- (210:2.5) node [currarrow, pos=0.5, sloped] {};
				\draw[line width=3, color=mycol4] (0:0) -- (330:2.5) node [currarrow, pos=0.5, sloped, rotate=180] {};
				\draw[line width=3, color=mycol4] (30:2.5) -- (90:2.5) node [currarrow, pos=0.5, sloped] {};
				\draw[line width=3, color=mycol4] (90:2.5) -- (150:2.5) node [currarrow, pos=0.5, sloped, rotate=180] {};
				\draw[line width=3, color=mycol4] (150:2.5) -- (210:2.5) node [currarrow, pos=0.5, sloped, rotate=180] {};
				\draw[line width=3, color=mycol4] (210:2.5) -- (270:2.5) node [currarrow, pos=0.5, sloped, , rotate=180] {};
				\draw[line width=3, color=mycol4] (270:2.5) -- (330:2.5) node [currarrow, pos=0.5, sloped] {};
				\draw[line width=3, color=mycol4] (330:2.5) -- (30:2.5) node [currarrow, pos=0.5, sloped] {};
				\fill[color=mycol3] (0:0) circle (0.2);
				\fill[color=mycol3] (30:2.5) circle (0.2);
				\fill[color=mycol3] (150:2.5) circle (0.2);
				\fill[color=mycol1] (270:2.5) circle (0.2);
				\fill[color=white] (270:2.5) circle (0.08);
				\fill[rotate around={90:(0,0)}, color=mycol2] (2.3, -0.2) rectangle (2.7, 0.2);
				\fill[rotate around={210:(0,0)}, color=mycol2] (2.3, -0.2) rectangle (2.7, 0.2);
				\fill[rotate around={330:(0,0)}, color=mycol2] (2.3, -0.2) rectangle (2.7, 0.2);
	\end{tikzpicture}}}
	\caption{Fractional vertices of $(7, 4)$, all with integrality gap $10/9$. Hollow circle: root; Circles: Terminals; Square: Steiner node. Note that the second, third, and fourth vertex belong to the same class of isomorphism, while the first one belongs to another class.}
	\label{fig:7-4}
\end{figure}
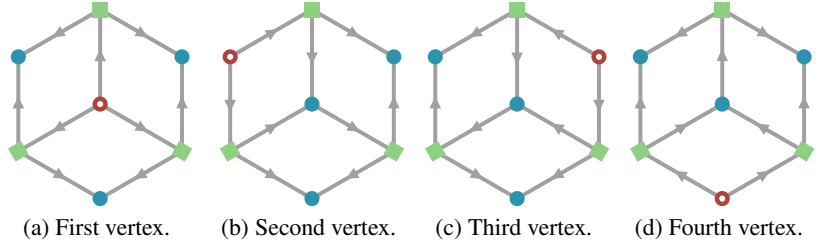

\paragraph{$(n = 8)$} The case $n = 8$ is more involved. 
The PHI does not find fractional vertices for the cases $t = 3, 4$. While the OTC does not find fractional vertices for $t=3$, it finds again the vertices of $(7,4)$ since it does not only find spanning vertices, as we already discussed. 
Both heuristics only find fractional vertices of integrality gap $1$ for the case $t=6$.
For $t=7$, only the PHI heuristic finds fractional vertices, all of them with an integrality gap of $1$. 
The most interesting case is $t=5$: the maximum integrality gap is depicted in \Cref{fig:max85} while different values of integrality gap are depicted in \Cref{fig:8-5-2} and \Cref{fig:8-5-3}. 
Note that the maximum integrality gap of this case for the PHI heuristic is $12/11$, while the maximum integrality gap for the OTC heuristic is again $10/9$: some of the vertices attaining this value are depicted in \Cref{fig:8-51}. 
Note also how these vertices can be obtained from vertices of $(7,4)$, as explained in \Cref{obs:adding-ones}.

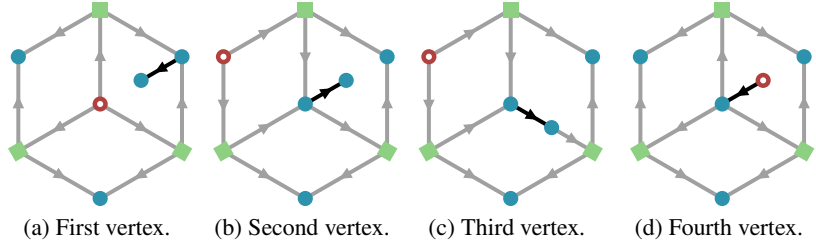
\begin{figure}[t!]
	\centering
	\subfloat[t!][First vertex.\label{fig:8-51-2}]{
		\scalebox{0.5}{\begin{tikzpicture}
				\draw[line width=3, color=mycol4] (0:0) -- (90:2.5) node [currarrow, pos=0.5, sloped] {};
				\draw[line width=3, color=mycol4] (0:0) -- (210:2.5) node [currarrow, pos=0.5, sloped, rotate=180] {};
				\draw[line width=3, color=mycol4] (0:0) -- (330:2.5) node [currarrow, pos=0.5, sloped] {};
				\draw[line width=3, color=mycol4] (30:2.5) -- (90:2.5) node [currarrow, pos=0.5, sloped] {};
				\draw[line width=3, color=mycol4] (90:2.5) -- (150:2.5) node [currarrow, pos=0.5, sloped, rotate=180] {};
				\draw[line width=3, color=mycol4] (150:2.5) -- (210:2.5) node [currarrow, pos=0.5, sloped, rotate=180] {};
				\draw[line width=3, color=mycol4] (210:2.5) -- (270:2.5) node [currarrow, pos=0.5, sloped] {};
				\draw[line width=3, color=mycol4] (270:2.5) -- (330:2.5) node [currarrow, pos=0.5, sloped, rotate=180] {};
				\draw[line width=3, color=mycol4] (330:2.5) -- (30:2.5) node [currarrow, pos=0.5, sloped] {};
				\draw[line width=3, color=black] (30:2.5) -- (30:1.25) node [currarrow, pos=0.5, sloped, rotate=180] {};
				\fill[color=mycol1] (0:0) circle (0.2);
				\fill[color=white] (0:0) circle (0.08);
				\fill[color=mycol3] (30:2.5) circle (0.2);
				\fill[color=mycol3] (30:1.25) circle (0.2);
				\fill[color=mycol3] (150:2.5) circle (0.2);
				\fill[color=mycol3] (270:2.5) circle (0.2);
				\fill[rotate around={90:(0,0)}, color=mycol2] (2.3, -0.2) rectangle (2.7, 0.2);
				\fill[rotate around={210:(0,0)}, color=mycol2] (2.3, -0.2) rectangle (2.7, 0.2);
				\fill[rotate around={330:(0,0)}, color=mycol2] (2.3, -0.2) rectangle (2.7, 0.2);
	\end{tikzpicture}}}\hspace*{0.1cm}
	\subfloat[t!][Second vertex.\label{fig:8-51-1}]{
		\scalebox{0.5}{\begin{tikzpicture}
				\draw[line width=3, color=mycol4] (0:0) -- (90:2.5) node [currarrow, pos=0.5, sloped, rotate=180] {};
				\draw[line width=3, color=mycol4] (0:0) -- (210:2.5) node [currarrow, pos=0.5, sloped] {};
				\draw[line width=3, color=mycol4] (0:0) -- (330:2.5) node [currarrow, pos=0.5, sloped] {};
				\draw[line width=3, color=mycol4] (30:2.5) -- (90:2.5) node [currarrow, pos=0.5, sloped] {};
				\draw[line width=3, color=mycol4] (90:2.5) -- (150:2.5) node [currarrow, pos=0.5, sloped] {};
				\draw[line width=3, color=mycol4] (150:2.5) -- (210:2.5) node [currarrow, pos=0.5, sloped] {};
				\draw[line width=3, color=mycol4] (210:2.5) -- (270:2.5) node [currarrow, pos=0.5, sloped] {};
				\draw[line width=3, color=mycol4] (270:2.5) -- (330:2.5) node [currarrow, pos=0.5, sloped, rotate=180] {};
				\draw[line width=3, color=mycol4] (330:2.5) -- (30:2.5) node [currarrow, pos=0.5, sloped] {};
				\draw[line width=3, color=black] (0:0) -- (30:1.25) node [currarrow, pos=0.5, sloped] {};
				\fill[color=mycol3] (0:0) circle (0.2);
				\fill[color=mycol3] (30:2.5) circle (0.2);
				\fill[color=mycol3] (30:1.25) circle (0.2);
				\fill[color=mycol1] (150:2.5) circle (0.2);
				\fill[color=white] (150:2.5) circle (0.08);
				\fill[color=mycol3] (270:2.5) circle (0.2);
				\fill[rotate around={90:(0,0)}, color=mycol2] (2.3, -0.2) rectangle (2.7, 0.2);
				\fill[rotate around={210:(0,0)}, color=mycol2] (2.3, -0.2) rectangle (2.7, 0.2);
				\fill[rotate around={330:(0,0)}, color=mycol2] (2.3, -0.2) rectangle (2.7, 0.2);
	\end{tikzpicture}}} \hspace*{0.1cm}
	\subfloat[t!][Third vertex. \label{fig:8-51-3}]{
		\scalebox{0.5}{\begin{tikzpicture}
				\draw[line width=3, color=mycol4] (0:0) -- (90:2.5) node [currarrow, pos=0.5, sloped, rotate=180] {};
				\draw[line width=3, color=mycol4] (0:0) -- (210:2.5) node [currarrow, pos=0.5, sloped] {};
				\draw[line width=3, color=mycol4] (330:1.25) -- (330:2.5) node [currarrow, pos=0.5, sloped] {};
				\draw[line width=3, color=mycol4] (30:2.5) -- (90:2.5) node [currarrow, pos=0.5, sloped] {};
				\draw[line width=3, color=mycol4] (90:2.5) -- (150:2.5) node [currarrow, pos=0.5, sloped] {};
				\draw[line width=3, color=mycol4] (150:2.5) -- (210:2.5) node [currarrow, pos=0.5, sloped] {};
				\draw[line width=3, color=mycol4] (210:2.5) -- (270:2.5) node [currarrow, pos=0.5, sloped] {};
				\draw[line width=3, color=mycol4] (270:2.5) -- (330:2.5) node [currarrow, pos=0.5, sloped, rotate=180] {};
				\draw[line width=3, color=mycol4] (330:2.5) -- (30:2.5) node [currarrow, pos=0.5, sloped] {};
				\draw[line width=3, color=black] (0:0) -- (330:1.25) node [currarrow, pos=0.5, sloped] {};
				\fill[color=mycol3] (0:0) circle (0.2);
				\fill[color=mycol3] (30:2.5) circle (0.2);
				\fill[color=mycol3] (330:1.25) circle (0.2);
				\fill[color=mycol1] (150:2.5) circle (0.2);
				\fill[color=white] (150:2.5) circle (0.08);
				\fill[color=mycol3] (270:2.5) circle (0.2);
				\fill[rotate around={90:(0,0)}, color=mycol2] (2.3, -0.2) rectangle (2.7, 0.2);
				\fill[rotate around={210:(0,0)}, color=mycol2] (2.3, -0.2) rectangle (2.7, 0.2);
				\fill[rotate around={330:(0,0)}, color=mycol2] (2.3, -0.2) rectangle (2.7, 0.2);
		\end{tikzpicture}}
	} \hspace*{0.1cm}
	\subfloat[t!][Fourth vertex.\label{fig:8-51-4}]{
		\scalebox{0.5}{\begin{tikzpicture}
				\draw[line width=3, color=mycol4] (0:0) -- (90:2.5) node [currarrow, pos=0.5, sloped] {};
				\draw[line width=3, color=mycol4] (0:0) -- (210:2.5) node [currarrow, pos=0.5, sloped, rotate=180] {};
				\draw[line width=3, color=mycol4] (0:0) -- (330:2.5) node [currarrow, pos=0.5, sloped] {};
				\draw[line width=3, color=mycol4] (30:2.5) -- (90:2.5) node [currarrow, pos=0.5, sloped] {};
				\draw[line width=3, color=mycol4] (90:2.5) -- (150:2.5) node [currarrow, pos=0.5, sloped, rotate=180] {};
				\draw[line width=3, color=mycol4] (150:2.5) -- (210:2.5) node [currarrow, pos=0.5, sloped, rotate=180] {};
				\draw[line width=3, color=mycol4] (210:2.5) -- (270:2.5) node [currarrow, pos=0.5, sloped] {};
				\draw[line width=3, color=mycol4] (270:2.5) -- (330:2.5) node [currarrow, pos=0.5, sloped, rotate=180] {};
				\draw[line width=3, color=mycol4] (330:2.5) -- (30:2.5) node [currarrow, pos=0.5, sloped] {};
				\draw[line width=3, color=black] (30:1.25) -- (0:0) node [currarrow, pos=0.5, sloped, rotate=180] {};
				\fill[color=mycol3] (0:0) circle (0.2);
				\fill[color=mycol3] (30:2.5) circle (0.2);
				\fill[color=mycol1] (30:1.25) circle (0.2);
				\fill[color=white] (30:1.25) circle (0.08);
				\fill[color=mycol3] (150:2.5) circle (0.2);
				\fill[color=mycol3] (270:2.5) circle (0.2);
				\fill[rotate around={90:(0,0)}, color=mycol2] (2.3, -0.2) rectangle (2.7, 0.2);
				\fill[rotate around={210:(0,0)}, color=mycol2] (2.3, -0.2) rectangle (2.7, 0.2);
				\fill[rotate around={330:(0,0)}, color=mycol2] (2.3, -0.2) rectangle (2.7, 0.2);
	\end{tikzpicture}}}
	\caption{Fractional vertices of $(8, 5)$, all with integrality gap 10/9. Hollow circle: root. Circles: Terminals. Squares: Steiner nodes.}
	\label{fig:8-51}
\end{figure}

\paragraph{$(n = 9)$} For $n = 9$, we can only run the PHI heuristics, as the OTC heuristic runs out of memory. 
Note how no vertices for the case $t=3, 4$ are found, accordingly to the second point of \Cref{lemma:num-edges}, while for the cases $t=7, 8$ only vertices of integrality gap $1$ are found. 
The maximum values of the integrality gap for $t=5, 6$ are $10/9$ and $14/13$, as shown in \Cref{fig:max95} and \Cref{fig:max9-6}, respectively. 
Different values of integrality gap are depicted in \Cref{fig:max-and-diff-gaps}. Notice how all the non-trivial values of integrality gaps found for $n \leq 9$ are of the form $\frac{2m}{2m-1}$, with $m = 5, 6, 7, 8, 9, 10, 11, 12$.
\paragraph{$(n \geq 10)$} For $n \geq 10$, even the PHI heuristic shows its limits. 
For $t \in \{8, 9\}$, the heuristic did not terminate within a timelit of 80 hours. 
The most interesting case we face is $t=6$, where we found a vertex with an integrality gap of $19/18$. In this case, the value is of the form $\frac{2m+1}{2m}$, in contrast to what we found for $n \leq 9$.

For $n = 11$ and $n = 12$, we show that the cases $ t \leq 5$ did not lead to any feasible PHI vertex. 
Tests with larger values of $t$ were computationally currently untractable.

Note that no vertices with an integrality gap greater than $1$ were found for the case $t=3$. The exactness for this case is discussed in~\cite[Section~2.7.1]{vicari2020simplexbasedsteinertree}.

\begin{table}[t!]
\centering
\caption{PHI versus OTC heuristic. The third and sixth columns report the number of non-isomorphic vertices; the fourth and seventh columns report the maximum value of the gap obtained for these vertices; the fifth and eighth columns report the number of vertices attaining the maximum gap. 
}
\label{tab:OTC_vs_PHI}
\begin{tabular}{ccrrrrrr}
	\toprule
	&  & \multicolumn{3}{c}{PHI} & \multicolumn{3}{c}{OTC}\\
	$n$  &  $t$ & \# vert. & max gap & \shortstack{\# vert.\\max. gap} & \# vert. & max gap & \shortstack{\# vert.\\max. gap}\\ 
	\midrule
	\multirow{2}{*}{6} 
	& 4 & 1 & 1 &  1 & 0 & - & - \\
	& 5 & 7 & 1 & 7 &  0 & - & - \\
	\midrule
	\multirow{3}{*}{7} 
	& 4 & 2 & 10/9 & 2 & 11 & 10/9 & 2 \\
	& 5 & 46 & 1 & 46 & 19 & 1 & 19 \\
	& 6 & 71 & 1 & 71 & 8 & 1 & 8\\
	\midrule
	\multirow{4}{*}{8} 
	&  4 & 0  & -  &  - & 19  & 10/9  & 2\\
	&  5 & 89  & 12/11  & 15  & 195  &  10/9 & 14 \\
	&  6 &  1\,070 & 1  & 1\,070  & 239  & 1  & 239\\
	&  7 &  758 & 1  &  758 & 0  & -  & -\\
	\bottomrule
\end{tabular}
\end{table}
\begin{table}[t!]
\centering
\caption{Partial performances of the PHI heuristic for $n \geq 9$.}
\label{tab:PHI>8}
\begin{tabular}{ccrrr}
	\toprule
	$n$ & $t$  & \# vert. & max gap & \shortstack{\# vert.\\max. gap} \\ 
	\midrule
	\multirow{4}{*}{9}  &  5 &  64 & 10/9  & 12  \\
	&  6 &   4\,389 & 14/13  & 200  \\
	&  7 & 21\,121  &  1 &  21\,121 \\
	&  8 &  8\,987 & 1  &  8\,987 \\
	\midrule
	\multirow{3}{*}{10} &  5 &  15 & 10/9  &  7 \\
	&  6 & 7\,386  &  10/9 & 73 \\
	& 7  &  155\,120  & 16/15 & 2\,653 \\
	\bottomrule
\end{tabular}
\end{table}

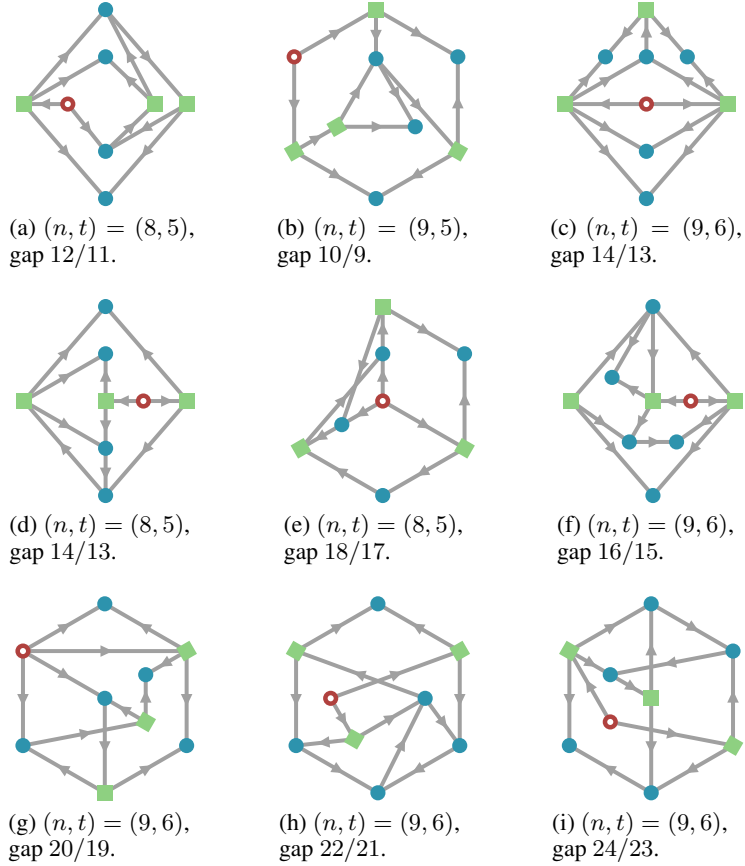
\begin{figure}[t!]
\centering
\subfloat[t!][$(n,t)=(8,5)$, gap $12/11$.\label{fig:max85}]{
	\scalebox{0.5}{\begin{tikzpicture}
			\draw[line width=3, color=mycol4] (90:2.5) -- (0:2.165) node [currarrow, pos=0.5, sloped] {};
			\draw[line width=3, color=mycol4] (0:2.165) -- (270:1.25) node [currarrow, pos=0.5, sloped, rotate=180] {};
			\draw[line width=3, color=mycol4] (0:2.165) -- (270:2.5) node [currarrow, pos=0.5, sloped, rotate=180] {};
			\draw[line width=3, color=mycol4] (180:1) -- (270:1.25) node [currarrow, pos=0.5, sloped] {};
			\draw[line width=3, color=mycol4] (180:1) -- (180:2.165) node [currarrow, pos=0.5, sloped, rotate=180] {};
			\draw[line width=3, color=mycol4] (180:2.165) -- (270:2.5) node [currarrow, pos=0.5, sloped] {};
			\draw[line width=3, color=mycol4] (180:2.165) -- (90:2.5) node [currarrow, pos=0.5, sloped] {};
			\draw[line width=3, color=mycol4] (180:2.165) -- (90:1.25) node [currarrow, pos=0.5, sloped] {};
			\draw[line width=3, color=mycol4] (270:1.25) -- (0:1.3) node [currarrow, pos=0.5, sloped] {};
			\draw[line width=3, color=mycol4] (0:1.3) -- (90:2.5) node [currarrow, pos=0.5, sloped, rotate=180] {};
			\draw[line width=3, color=mycol4] (0:1.3) -- (90:1.25) node [currarrow, pos=0.5, sloped, rotate=180] {};
			\fill[color=mycol3] (90:2.5) circle (0.2);
			\fill[color=mycol3] (270:2.5) circle (0.2);
			\fill[color=mycol3] (90:1.25) circle (0.2);
			\fill[color=mycol3] (270:1.25) circle (0.2);
			\fill[color=mycol1] (180:1) circle (0.2);
			\fill[color=white] (180:1) circle (0.08);
			\fill[rotate around={180:(0,0)}, color=mycol2] (1.965, -0.2) rectangle (2.365, 0.2);
			\fill[rotate around={0:(0,0)}, color=mycol2] (1.965, -0.2) rectangle (2.365, 0.2);
			\fill[rotate around={0:(0,0)}, color=mycol2] (1.1, -0.2) rectangle (1.5, 0.2);
\end{tikzpicture}}} \hspace*{1cm}
\subfloat[t!][$(n,t)=(9,5)$, gap $10/9$.\label{fig:max95}]{
	\scalebox{0.5}{\begin{tikzpicture}
			\draw[line width=3, color=mycol4] (90:1.2) -- (90:2.5) node [currarrow, pos=0.5, sloped, rotate=180] {};
			\draw[line width=3, color=mycol4] (210:2.5) -- (210:1.2) node [currarrow, pos=0.5, sloped] {};
			\draw[line width=3, color=mycol4] (90:1.2) -- (330:2.5) node [currarrow, pos=0.5, sloped] {};
			\draw[line width=3, color=mycol4] (90:1.2) -- (330:1.2) node [currarrow, pos=0.5, sloped] {};
			\draw[line width=3, color=mycol4] (210:1.2) -- (330:1.2) node [currarrow, pos=0.5, sloped] {};
			\draw[line width=3, color=mycol4] (210:1.2) -- (90:1.2) node [currarrow, pos=0.5, sloped] {};
			\draw[line width=3, color=mycol4] (30:2.5) -- (90:2.5) node [currarrow, pos=0.5, sloped] {};
			\draw[line width=3, color=mycol4] (90:2.5) -- (150:2.5) node [currarrow, pos=0.5, sloped] {};
			\draw[line width=3, color=mycol4] (150:2.5) -- (210:2.5) node [currarrow, pos=0.5, sloped] {};
			\draw[line width=3, color=mycol4] (210:2.5) -- (270:2.5) node [currarrow, pos=0.5, sloped] {};
			\draw[line width=3, color=mycol4] (270:2.5) -- (330:2.5) node [currarrow, pos=0.5, sloped, rotate=180] {};
			\draw[line width=3, color=mycol4] (330:2.5) -- (30:2.5) node [currarrow, pos=0.5, sloped] {};
			\fill[color=mycol3] (90:1.2) circle (0.2);
			\fill[color=mycol3] (30:2.5) circle (0.2);
			\fill[color=mycol1] (150:2.5) circle (0.2);
			\fill[color=white] (150:2.5) circle (0.08);
			\fill[color=mycol3] (270:2.5) circle (0.2);
			\fill[color=mycol3] (330:1.2) circle (0.2);
			\fill[rotate around={90:(0,0)}, color=mycol2] (2.3, -0.2) rectangle (2.7, 0.2);
			\fill[rotate around={210:(0,0)}, color=mycol2] (2.3, -0.2) rectangle (2.7, 0.2);
			\fill[rotate around={330:(0,0)}, color=mycol2] (2.3, -0.2) rectangle (2.7, 0.2);
			\fill[rotate around={210:(0,0)}, color=mycol2] (1.0, -0.2) rectangle (1.4, 0.2);
\end{tikzpicture}}} \hspace*{1cm}
\subfloat[t!][$(n,t)=(9,6)$, gap $14/13$. \label{fig:max9-6}]{
	\scalebox{0.5}{\begin{tikzpicture}
			\draw[line width=3, color=mycol4] (0:0) -- (0:2.165) node [currarrow, pos=0.5, sloped] {};
			\draw[line width=3, color=mycol4] (0:0) -- (180:2.165) node [currarrow, pos=0.5, sloped, rotate=180] {};
			\draw[line width=3, color=mycol4] (180:2.165) -- (0, 1.25) node [currarrow, pos=0.5, sloped] {};
			\draw[line width=3, color=mycol4] (180:2.165) -- (-1.0825, 1.25) node [currarrow, pos=0.5, sloped] {};
			\draw[line width=3, color=mycol4] (180:2.165) -- (0, -1.25) node [currarrow, pos=0.5, sloped] {};
			\draw[line width=3, color=mycol4] (180:2.165) -- (0, -2.5) node [currarrow, pos=0.5, sloped] {};
			\draw[line width=3, color=mycol4] (0:2.165) -- (0, 1.25) node [currarrow, pos=0.5, sloped, rotate=180] {};
			\draw[line width=3, color=mycol4] (0:2.165) -- (1.0825, 1.25) node [currarrow, pos=0.5, sloped, rotate=180] {};
			\draw[line width=3, color=mycol4] (0:2.165) -- (0, -1.25) node [currarrow, pos=0.5, sloped, rotate=180] {};
			\draw[line width=3, color=mycol4] (0:2.165) -- (0, -2.5) node [currarrow, pos=0.5, sloped, rotate=180] {};
			\draw[line width=3, color=mycol4] (0, 1.25) -- (0, 2.5) node [currarrow, pos=0.5, sloped] {};
			\draw[line width=3, color=mycol4] (0, 2.5) -- (1.0825, 1.25) node [currarrow, pos=0.5, sloped] {};
			\draw[line width=3, color=mycol4] (0, 2.5) -- (-1.0825, 1.25) node [currarrow, pos=0.5, sloped, rotate=180] {};
			\fill[color=mycol3] (270:2.5) circle (0.2);
			\fill[color=mycol3] (1.0825, 1.25) circle (0.2);
			\fill[color=mycol3] (-1.0825, 1.25) circle (0.2);
			\fill[color=mycol3] (90:1.25) circle (0.2);
			\fill[color=mycol3] (270:1.25) circle (0.2);
			\fill[color=mycol1] (0:0) circle (0.2);
			\fill[color=white] (0:0) circle (0.08);
			\fill[rotate around={180:(0,0)}, color=mycol2] (1.965, -0.2) rectangle (2.365, 0.2);
			\fill[rotate around={0:(0,0)}, color=mycol2] (1.965, -0.2) rectangle (2.365, 0.2);
			\fill[color=mycol2] (-0.2, 2.3) rectangle (0.2, 2.7);
	\end{tikzpicture}}
}\\
\subfloat[t!][$(n,t)=(8,5)$, \\gap $14/13$.\label{fig:8-5-2}]{
	\scalebox{0.5}{\begin{tikzpicture}
			\draw[line width=3, color=mycol4] (0:1) -- (0:2.165) node [currarrow, pos=0.5, sloped] {};
			\draw[line width=3, color=mycol4] (0:1) -- (0:0) node [currarrow, pos=0.5, sloped, rotate=180] {};
			\draw[line width=3, color=mycol4] (0:2.165) -- (90:2.5) node [currarrow, pos=0.5, sloped, rotate=180] {};
			\draw[line width=3, color=mycol4] (0:2.165) -- (270:2.5) node [currarrow, pos=0.5, sloped, rotate=180] {};
			\draw[line width=3, color=mycol4] (0:0) -- (90:1.25) node [currarrow, pos=0.5, sloped] {};
			\draw[line width=3, color=mycol4] (0:0) -- (270:1.25) node [currarrow, pos=0.5, sloped] {};
			\draw[line width=3, color=mycol4] (270:1.25) -- (270:2.5) node [currarrow, pos=0.5, sloped] {};
			\draw[line width=3, color=mycol4] (270:2.5) -- (180:2.165) node [currarrow, pos=0.5, sloped, rotate=180] {};
			\draw[line width=3, color=mycol4] (180:2.165) -- (270:1.25) node [currarrow, pos=0.5, sloped] {};
			\draw[line width=3, color=mycol4] (180:2.165) -- (90:1.25) node [currarrow, pos=0.5, sloped] {};
			\draw[line width=3, color=mycol4] (180:2.165) -- (90:2.5) node [currarrow, pos=0.5, sloped] {};
			\fill[color=mycol3] (90:2.5) circle (0.2);
			\fill[color=mycol3] (270:2.5) circle (0.2);
			\fill[color=mycol3] (90:1.25) circle (0.2);
			\fill[color=mycol3] (270:1.25) circle (0.2);
			\fill[color=mycol1] (0:1) circle (0.2);
			\fill[color=white] (0:1) circle (0.08);
			\fill[rotate around={180:(0,0)}, color=mycol2] (1.965, -0.2) rectangle (2.365, 0.2);
			\fill[rotate around={0:(0,0)}, color=mycol2] (1.965, -0.2) rectangle (2.365, 0.2);
			\fill[rotate around={0:(0,0)}, color=mycol2] (-0.2, -0.2) rectangle (0.2, 0.2);
	\end{tikzpicture}}
} \hspace*{1cm}
\subfloat[t!][$(n,t)=(8,5)$, \\gap $18/17$. \label{fig:8-5-3}]{
	\scalebox{0.5}{\begin{tikzpicture}
			\draw[line width=3, color=mycol4] (0:0) -- (90:1.25) node [currarrow, pos=0.5, sloped] {};
			\draw[line width=3, color=mycol4] (0:0) -- (210:1.25) node [currarrow, pos=0.5, sloped, rotate=180] {};
			\draw[line width=3, color=mycol4] (0:0) -- (330:2.5) node [currarrow, pos=0.5, sloped] {};
			\draw[line width=3, color=mycol4] (90:1.25) -- (90:2.5)  node [currarrow, pos=0.5, sloped] {};
			\draw[line width=3, color=mycol4] (210:1.25) -- (210:2.5) node [currarrow, pos=0.5, sloped, rotate=180] {};
			\draw[line width=3, color=mycol4] (330:2.5) -- (30:2.5) node [currarrow, pos=0.5, sloped] {};
			\draw[line width=3, color=mycol4] (330:2.5) -- (270:2.5) node [currarrow, pos=0.5, sloped, rotate=180] {};
			\draw[line width=3, color=mycol4] (90:2.5) -- (30:2.5) node [currarrow, pos=0.5, sloped] {};
			\draw[line width=3, color=mycol4] (210:2.5) -- (270:2.5) node [currarrow, pos=0.5, sloped, rotate=180] {};
			\draw[line width=3, color=mycol4] (90:2.5) -- (210:1.25) node [currarrow, pos=0.5, sloped, rotate=180] {};
			\draw[line width=3, color=mycol4] (210:2.5) -- (90:1.25) node [currarrow, pos=0.5, sloped] {};
			\fill[color=mycol1] (0:0) circle (0.2);
			\fill[color=white] (0:0) circle (0.08);
			\fill[color=mycol3] (30:2.5) circle (0.2);
			\fill[color=mycol3] (90:1.25) circle (0.2);
			\fill[color=mycol3] (210:1.25) circle (0.2);
			\fill[color=mycol3] (270:2.5) circle (0.2);
			\fill[rotate around={90:(0,0)}, color=mycol2] (2.3, -0.2) rectangle (2.7, 0.2);
			\fill[rotate around={210:(0,0)}, color=mycol2] (2.3, -0.2) rectangle (2.7, 0.2);
			\fill[rotate around={330:(0,0)}, color=mycol2] (2.3, -0.2) rectangle (2.7, 0.2);
\end{tikzpicture}}}
\hspace*{1cm}
\subfloat[t][$(n,t)=(9,6)$, \\gap $16/15$.]{
	\scalebox{0.5}{\begin{tikzpicture}
			\draw[line width=3, color=mycol4] (0:1) -- (0:2.165) node [currarrow, pos=0.5, sloped] {};
			\draw[line width=3, color=mycol4] (0:1) -- (0:0) node [currarrow, pos=0.5, sloped, rotate=180] {};
			\draw[line width=3, color=mycol4] (0:2.165) -- (300:1.25) node [currarrow, pos=0.5, sloped, rotate=180] {};
			\draw[line width=3, color=mycol4] (0:2.165) -- (90:2.5) node [currarrow, pos=0.5, sloped, rotate=180] {};
			\draw[line width=3, color=mycol4] (0:2.165) -- (270:2.5) node [currarrow, pos=0.5, sloped, rotate=180] {};
			\draw[line width=3, color=mycol4] (0:0) -- (150:1.25) node [currarrow, pos=0.5, sloped, rotate=180] {};
			\draw[line width=3, color=mycol4] (0:0) -- (240:1.25) node [currarrow, pos=0.5, sloped, rotate=180] {};
			\draw[line width=3, color=mycol4] (180:2.165) -- (270:2.5) node [currarrow, pos=0.5, sloped] {};
			\draw[line width=3, color=mycol4] (180:2.165) -- (240:1.25) node [currarrow, pos=0.5, sloped] {};
			\draw[line width=3, color=mycol4] (90:2.5) -- (150:1.25) node [currarrow, pos=0.5, sloped, rotate=180] {};
			\draw[line width=3, color=mycol4] (0:0) -- (90:2.5) node [currarrow, pos=0.5, sloped, rotate=180] {};
			\draw[line width=3, color=mycol4] (180:2.165) -- (90:2.5) node [currarrow, pos=0.5, sloped] {};
			\draw[line width=3, color=mycol4] (240:1.25) -- (300:1.25) node [currarrow, pos=0.5, sloped] {};
			\fill[color=mycol3] (90:2.5) circle (0.2);
			\fill[color=mycol3] (270:2.5) circle (0.2);
			\fill[color=mycol3] (150:1.25) circle (0.2);
			\fill[color=mycol3] (240:1.25) circle (0.2);
			\fill[color=mycol3] (300:1.25) circle (0.2);
			\fill[color=mycol1] (0:1) circle (0.2);
			\fill[color=white] (0:1) circle (0.08);
			\fill[rotate around={180:(0,0)}, color=mycol2] (1.965, -0.2) rectangle (2.365, 0.2);
			\fill[rotate around={0:(0,0)}, color=mycol2] (1.965, -0.2) rectangle (2.365, 0.2);
			\fill[rotate around={0:(0,0)}, color=mycol2] (-0.2, -0.2) rectangle (0.2, 0.2);
\end{tikzpicture}}}
\\
\subfloat[t][$(n,t)=(9,6)$, \\gap $20/19$.]{
	\scalebox{0.5}{\begin{tikzpicture}
			\draw[line width=3, color=mycol4] (150:2.5) -- (90:2.5) node [currarrow, pos=0.5, sloped] {};
			\draw[line width=3, color=mycol4] (150:2.5) -- (150:0) node [currarrow, pos=0.5, sloped] {};
			\draw[line width=3, color=mycol4] (150:2.5) -- (210:2.5) node [currarrow, pos=0.5, sloped] {};
			\draw[line width=3, color=mycol4] (150:2.5) -- (30:2.5) node [currarrow, pos=0.5, sloped] {};
			\draw[line width=3, color=mycol4] (30:2.5) -- (90:2.5) node [currarrow, pos=0.5, sloped, rotate=180] {};
			\draw[line width=3, color=mycol4] (30:2.5) -- (30:1.25) node [currarrow, pos=0.5, sloped, rotate=180] {};
			\draw[line width=3, color=mycol4] (30:2.5) -- (330:2.5) node [currarrow, pos=0.5, sloped] {};
			\draw[line width=3, color=mycol4] (270:2.5) -- (210:2.5) node [currarrow, pos=0.5, sloped, rotate=180] {};
			\draw[line width=3, color=mycol4] (330:1.25) -- (30:1.25) node [currarrow, pos=0.5, sloped] {};
			\draw[line width=3, color=mycol4] (330:1.25) -- (150:0) node [currarrow, pos=0.5, sloped, rotate=180] {};
			\draw[line width=3, color=mycol4] (210:2.5) -- (330:1.25) node [currarrow, pos=0.5, sloped] {};
			\draw[line width=3, color=mycol4] (270:2.5) -- (330:2.5) node [currarrow, pos=0.5, sloped] {};
			\draw[line width=3, color=mycol4] (150:0) -- (270:2.5) node [currarrow, pos=0.5, sloped] {};
			\fill[color=mycol3] (150:0) circle (0.2);
			\fill[color=mycol3] (30:1.25) circle (0.2);
			\fill[color=mycol3] (90:2.5) circle (0.2);
			\fill[color=mycol3] (330:2.5) circle (0.2);
			\fill[color=mycol1] (150:2.5) circle (0.2);
			\fill[color=white] (150:2.5) circle (0.08);
			\fill[color=mycol3] (210:2.5) circle (0.2);
			\fill[rotate around={30:(0,0)}, color=mycol2] (2.3, -0.2) rectangle (2.7, 0.2);
			\fill[rotate around={270:(0,0)}, color=mycol2] (2.3, -0.2) rectangle (2.7, 0.2);
			\fill[rotate around={330:(0,0)}, color=mycol2] (1.05, -0.2) rectangle (1.45, 0.2);
\end{tikzpicture}}}
\hspace*{1cm}
\subfloat[t][$(n,t)=(9,6)$, \\gap $22/21$.]{
	\scalebox{0.5}{\begin{tikzpicture}
			\draw[line width=3, color=mycol4] (150:2.5) -- (90:2.5) node [currarrow, pos=0.5, sloped] {};
			\draw[line width=3, color=mycol4] (30:2.5) -- (90:2.5) node [currarrow, pos=0.5, sloped, rotate=180] {};
			\draw[line width=3, color=mycol4] (150:2.5) -- (210:2.5) node [currarrow, pos=0.5, sloped] {};
			\draw[line width=3, color=mycol4] (180:1.25) -- (240:1.25) node [currarrow, pos=0.5, sloped] {};
			\draw[line width=3, color=mycol4] (180:1.25) -- (30:2.5) node [currarrow, pos=0.5, sloped] {};
			\draw[line width=3, color=mycol4] (240:1.25) -- (210:2.5) node [currarrow, pos=0.5, sloped, rotate=180] {};
			\draw[line width=3, color=mycol4] (240:1.25) -- (0:1.25) node [currarrow, pos=0.5, sloped] {};
			\draw[line width=3, color=mycol4] (0:1.25) -- (150:2.5) node [currarrow, pos=0.5, sloped, rotate=180] {};
			\draw[line width=3, color=mycol4] (0:1.25) -- (330:2.5) node [currarrow, pos=0.5, sloped] {};
			\draw[line width=3, color=mycol4] (270:2.5) -- (0:1.25) node [currarrow, pos=0.5, sloped] {};
			\draw[line width=3, color=mycol4] (330:2.5) -- (270:2.5) node [currarrow, pos=0.5, sloped, rotate=180] {};
			\draw[line width=3, color=mycol4] (210:2.5) -- (270:2.5) node [currarrow, pos=0.5, sloped] {};
			\draw[line width=3, color=mycol4] (30:2.5) -- (330:2.5) node [currarrow, pos=0.5, sloped] {};
			\fill[color=mycol3] (90:2.5) circle (0.2);
			\fill[color=mycol3] (210:2.5) circle (0.2);
			\fill[color=mycol3] (330:2.5) circle (0.2);
			\fill[color=mycol1] (180:1.25) circle (0.2);
			\fill[color=white] (180:1.25) circle (0.08);
			\fill[color=mycol3] (270:2.5) circle (0.2);
			\fill[color=mycol3] (0:1.25) circle (0.2);
			\fill[rotate around={150:(0,0)}, color=mycol2] (2.3, -0.2) rectangle (2.7, 0.2);
			\fill[rotate around={240:(0,0)}, color=mycol2] (1.05, -0.2) rectangle (1.45, 0.2);
			\fill[rotate around={30:(0,0)}, color=mycol2] (2.3, -0.2) rectangle (2.7, 0.2);
\end{tikzpicture}}}
\hspace*{1cm}
\subfloat[t][$(n,t)=(9,6)$, \\gap $24/23$.]{
	\scalebox{0.5}{\begin{tikzpicture}
			\draw[line width=3, color=mycol4] (210:1.25) -- (330:2.5) node [currarrow, pos=0.5, sloped] {};
			\draw[line width=3, color=mycol4] (210:1.25) -- (150:2.5) node [currarrow, pos=0.5, sloped, rotate=180] {};
			\draw[line width=3, color=mycol4] (150:2.5) -- (210:2.5) node [currarrow, pos=0.5, sloped] {};
			\draw[line width=3, color=mycol4] (150:2.5) -- (90:2.5) node [currarrow, pos=0.5, sloped] {};
			\draw[line width=3, color=mycol4] (150:2.5) -- (150:1.25) node [currarrow, pos=0.5, sloped] {};
			\draw[line width=3, color=mycol4] (150:1.25) -- (0:0) node [currarrow, pos=0.5, sloped] {};
			\draw[line width=3, color=mycol4] (30:2.5) -- (150:1.25) node [currarrow, pos=0.5, sloped, rotate=180] {};
			\draw[line width=3, color=mycol4] (0:0) -- (90:2.5) node [currarrow, pos=0.5, sloped] {};
			\draw[line width=3, color=mycol4] (30:2.5) -- (90:2.5) node [currarrow, pos=0.5, sloped] {};
			\draw[line width=3, color=mycol4] (0:0) -- (270:2.5) node [currarrow, pos=0.5, sloped] {};
			\draw[line width=3, color=mycol4] (330:2.5) -- (30:2.5) node [currarrow, pos=0.5, sloped] {};
			\draw[line width=3, color=mycol4] (330:2.5) -- (270:2.5) node [currarrow, pos=0.5, sloped, rotate=180] {};
			\draw[line width=3, color=mycol4] (270:2.5) -- (210:2.5) node [currarrow, pos=0.5, sloped, rotate=180] {};
			\fill[color=mycol3] (30:2.5) circle (0.2);
			\fill[color=mycol3] (90:2.5) circle (0.2);
			\fill[color=mycol3] (210:2.5) circle (0.2);
			\fill[color=mycol1] (210:1.25) circle (0.2);
			\fill[color=white] (210:1.25) circle (0.08);
			\fill[color=mycol3] (150:1.25) circle (0.2);
			\fill[color=mycol3] (270:2.5) circle (0.2);
			\fill[rotate around={0:(0,0)}, color=mycol2] (-0.2, -0.2) rectangle (0.2, 0.2);
			\fill[rotate around={150:(0,0)}, color=mycol2] (2.3, -0.2) rectangle (2.7, 0.2);
			\fill[rotate around={330:(0,0)}, color=mycol2] (2.3, -0.2) rectangle (2.7, 0.2);
\end{tikzpicture}}}

\caption{Fractional vertices of different gaps for different values of $(n,t)$. The first three vertices attain the maximum gap for their respective value of $(n,t)$.}
\label{fig:max-and-diff-gaps}
\end{figure}%

\subsection{A comparison between the two proposed heuristics}
\label{subsec:OTCvsPHI}
In this subsection, we discuss an in-depth comparison between the PHI and OTC heuristics.
First, notice how neither of the two are exhaustive algorithms: at least one vertex can be found by the OTC heuristic but not by the PHI heuristic, and vice versa (see Figures \Cref{fig:8-51-1} and \Cref{obs:1-2-not-exha}).
While the PHI heuristic is tailored for vertices with particular values, and so with a particular structure, the OTC is general enough to find different types of vertices; moreover, it remains an open question whether the heuristic becomes an exhaustive search by dropping the connectivity constraint.  

Computationally, the OTC heuristic is highly demanding, even when limited to generating only connected graphs. For each generated graph, all possible assignments of the root, terminal nodes, and Steiner nodes must be considered, and an LP must be solved for every assignment.
Moreover, there is no guarantee that the solution to the LP will be fractional; in fact, it may correspond to an equivalent integer solution.
In addition, the OTC heuristic does not ensure that the generated solutions are non-isomorphic, necessitating a post-processing step to filter out isomorphic graphs based on node-colored edge-weighted graph isomorphism.
The algorithm does not even guarantee finding spanning vertices. 
The PHI heuristic doesn't generate isomorphic graphs, and hence, every vertex generated belongs to a unique class of isomorphism. 
In addition, no LP needs to be solved since, given the orientation of the arcs, the role of every node is uniquely determined.
Lastly, note that in the OTC heuristic, we have applied the extra bounds on the number of edges $n\cdot t - t^2$ derived after a first set of computational experiments. Without this hypothesis, OTC is untractable for $n \geq 8$. 
The results of \Cref{tab:OTC_vs_PHI} are obtained with this extra constraint. 
Note also that, even with the aforementioned bound on the number of edges, OTC is untractable for $ n \geq 9$.

\subsection{Beyond pure half-integer vertices}\label{subsec:beyondphi}
The computational results of the previous subsection show that the PHI heuristic is better than OTC in finding interesting vertices of the CM polytope. 
In addition, the PHI heuristic can be extended to enumerate all the vertices of the type $\{0, 1/m\}$, $m \in \mathbb{N}_{\geq 3}$.
For example, an interesting case is $m = 4$, namely, when the vertices take value only in the set $\{0, 1/4\}$.
Let us call these vertices \emph{pure one-quarter} (POQ) vertices.
In~\cite{konemann2011partitionbased}, the authors show that the integrality gap of the BCR formulation is at least $\frac{8}{7}$ by exposing an instance leading to such a gap. 
The instance has 15 nodes and 8 terminals.
The optimal vertex is of POQ type, and it originated from a personal communication between Martin Skutella and the authors of~\cite{konemann2011partitionbased}. This makes POQ vertices particularly relevant for our study.
\Cref{fig:15-8} shows Skutella's graph.
Note that solving the Gap function for the CM formulation leads to a gap equal to $\frac{8}{7}$. 
Hence, the maximum integrality gap for Skutella's graph is exactly $\frac{8}{7}$. 

We defined a modified version of the PHI algorithm to find POQ vertices (for further details, see \Cref{app:POQ}), but we were not able to find any vertex with the above properties before the computation became intractable, that is $n \geq 8$.

\section{Conclusion and future works}
In this paper, we have studied the metric STP on graphs, focusing on computing lower bounds for the integrality gap for the BCR and the CM formulations. 
We introduced a novel ILP formulation, the Complete Metric (CM) model, tailored for the complete metric Steiner tree problem. This formulation overcomes the limitations of the BCR formulation in the metric case.
For the CM formulation, we prove several structural properties of the polytope associated with its natural LP relaxation.

The core of our contribution presented in this paper is extending the Gap problem introduced in~\cite{boyd2005computing} for the symmetric TSP, to the metric Steiner tree problem. This approach relies entirely on the enumeration of the vertices of the polytope. However, this task is significantly more challenging in the case of the Steiner Tree Problem compared to the TSP. Therefore, after restricting the search space and establishing certain properties of the relevant vertices in the STP, we designed two heuristics for partial vertices enumeration.
Our heuristics outperform the exact method obtained as a natural extension of~\cite{boyd2005computing} to the Steiner problem and can generate nontrivial vertices for $n$ up to 10. Note that exact methods got stuck already for $n = 6$.

We compare the performances of the two heuristics and their impact on providing insights into the exact value of the integrality gap. Although we cannot improve the bound of $\frac{10}{9}$ with $n \leq 10$, we find different structures of vertices leading to non-trivial gaps.
By directly exploring vertices similar to those yielding the highest gaps for $n > 10$, we observed that these structures cannot be present for small values of $n$. 
Hence, we conjecture that with $n \leq 10$, the highest gap is $\frac{10}{9}$.

We retain that our study raises several interesting research questions. First, can an ad-hoc branching procedure be designed for the CM formulation? Second, can we improve the OTC heuristic by reducing the number of combinations we have to analyze without losing any of the outputs? Third, can we prove any further characterization of the vertices that reduce the effort for Polymake, similarly to what has been done in~\cite{benoit2008finding,boyd2007structure,elliott2008integrality}? Lastly, can we enhance the design and implementation of the POQ heuristic to explore whether new lower bounds for the integrality gap are achievable for this type of vertices in higher dimensions?

We conclude this paper with two conjectures:
first, we conjecture that our OTC procedure, without the restriction on connectedness and the bound on the number of edges, is exhaustive and, hence, every vertex of $P_{CM}(n, t)$ can be obtained as an optimal solution of a $\{1, 2\}$-cost instance. 
Second, we conjecture that every spanning vertex $x$ of $P_{CM}(n, t)$ with $x_{ij} \in \{0, 1/m\}$, $m \geq 2$, has an in-degree of $1$ in every Steiner node, and hence our PHI is exhaustive for every pure half-integer spanning vertex.
\section*{Acknowledgments}
 Stefano Gualandi acknowledges the contribution of the National Recovery and Resilience Plan, Mission 4 Component 2 - Investment 1.4 - NATIONAL CENTER FOR HPC, BIG DATA AND QUANTUM COMPUTING (project code: CN\_00000013) - funded by the European Union - NextGenerationEU. The work of Ambrogio Maria Bernardelli is supported by a Ph.D. scholarship funded under the``Programma Operativo Nazionale Ricerca e Innovazione'' 2014--2020. The work of Eleonora Vercesi is supported by the Swiss National Science Foundation project n. 200021\_212929 / 1 ``Computational methods for integrality gaps analysis'', project code: 36RAGAP

\setcounter{biburllcpenalty}{7000}
\setcounter{biburlucpenalty}{8000}
\printbibliography
\newpage
\appendix

\section{Making use of the Multi Commodity Flow formulation}

\subsection{Proof of \texorpdfstring{\Cref{lemma:constr-redundancy}}{Lemma 3.2}} \label{app:lemma-proof}

Before proving \Cref{lemma:constr-redundancy}, we recall the  \emph{Multi Commodity Flow} (MCF) formulation~\cite{chopra2001polyhedral}:

\begin{subequations}
	\begin{align}
		\label{eq:mcf} \min & \sum_{\{i, j\} \in E} c_e (x_{ij} + x_{ji}) & \\    
		\mbox{s.t.} \quad & x_{i j}+x_{j i} \leq 1, & e=\{i, j\} \in E, \label{eq:mcf_degree} \\
		&  f^t\left(\delta^{-}(v)\right) -  f^t\left(\delta^{+}(v)\right) = 
		\begin{cases}
			-1, \; \text{ if } v = r \\
			1, \; \text{ if } v = t\\
			0, \; \text{ otherwise,}
		\end{cases} & v \in V, \; t \in T \setminus \{r\},\label{eq:mcf_flow}\\
		&  f_{ij}^t \leq  x_{ij} & \forall (i,j) \in A,\label{eq:mcf_limit}\\
		&  f_{ij}^t, x_{ij} \in \{0, 1\}, & \forall (i,j) \in A. \label{eq:mcf_01}
	\end{align}
\end{subequations}

This formulation has a few interesting properties that we use in the following proof of \Cref{lemma:constr-redundancy}.
\begin{proof}[Proof of \Cref{lemma:constr-redundancy}]
	Let $x_{ij}$ be an optimum vertex for the BCR formulation with a positive cost $c$. By~\cite[Theorem~3.2]{chopra2001polyhedral}, it follows
	\begin{equation*}
		\min \{c^\intercal x \, | \, x \in P_{MCF}(n,t)_{|x}\} = \min \{c^ \intercal  x \, | \, x \in P_{BCR}(n,t)\}.
	\end{equation*}
	Hence, there exists a configuration of variables $f_{ij}^t$ such that $f_{ij}^t \leq x_{ij}$ for every $t \in T, \; i,j \in V$ and $\sum_i f_{ij}^t - \sum_i f_{ji}^t = 0$ for every $t \in T, \; j \in V \setminus T$. 
	Because $x_{ij}$ is optimal for strictly positive costs, we have that $x_{ij} = \max_{t \in T} f_{ij}^t$ and so there exists $t_{ij} \in T$ such that $x_{ij} = f_{ij}^{t_{ij}}$. 
	Now, let $k \in V \setminus T$. For every $a \in \delta^+(k)$, that is, for every $l \in V\setminus \{k\}$ we have that 
	\begin{align*}
		x_a & = x_{kl}                      & \text{by definition} \\
		& = f_{kl}^{t_{kl}}             & \text{by maximization }  \\
		& \leq \sum_i f_{ki}^{t_{kl}}   & \text{by nonnegativity} \\
		& = \sum_i f_{ik}^{t_{kl}}      & \text{by \eqref{eq:mcf_flow}} \\
		& \leq \sum_i x_{ik}            & \text{by \eqref{eq:mcf_limit}} \\
		& = x(\delta^-(k))              & \text{by definition}
	\end{align*}
	which is equivalent to Constraint \eqref{eq:scip_inout2}.
\end{proof}
\subsection{Details on PHI spanning vertices}\label{app:PHI-sp}
We conjecture that, for every PHI spanning vertex, every Steiner node has indegree exactly one. We believe this property is true because of the following reasoning: first of all, because of \Cref{lemma:oriented-cycles}, there are no loops of length $2$, and so every edge can be oriented in only one way. Suppose there exists a Steiner node $k$ such that $indeg(k) > 1$, and since the maximum inflow is $1$ because of Constraint \eqref{eq:cm_inflow} and we are dealing with pure half-integer solutions, we have that $indeg(k) = 2$. Then, regarding the MCF formulation, there exist $T_1, T_2 \subset T$, $T_1, T_2 \neq \emptyset$, and $i, j \in V$ such that $f_{ik}^{t_1} = f_{jk}^{t_2} = 1/2$, $\forall t_1 \in T_1, t_2 \in T_2$. We conjecture that is always possible to construct $y, z \in P_{BCR}(n,t)$ such that $y \neq z$ and $x = \frac{1}{2}y + \frac{1}{2} z$, leading to a contradiction. In particular, $y$ is derived by $x$ by setting $f_{ik}^{t_1} = 1, \;  f_{ik}^{t_2} = 0$, $\forall t_1 \in T_1, t_2 \in T_2$, and all the other variables are set accordingly to \eqref{eq:mcf_flow}, while $z$ is derived by $x$ by setting $f_{ik}^{t_1} = 0, \;  f_{ik}^{t_2} = 1$, $\forall t_1 \in T_1, t_2 \in T_2$, and all the other variables are set accordingly to \eqref{eq:mcf_flow}.

\section{Enumerating vertices with Polymake} \label{app:polymake}
\subsection{Enumerating vertices with Polymake and weakness of the BCR formulation}
\label{sec:gap_DCUT}

As discussed in the previous section, we aim to solve the Gap problem on every vertex.
Hence, we need an exhaustive list of vertices of the polytope $P_{BCR}(n,t)$ for each $n \geq 3$, for each $ 3 \leq t \leq n - 1$.
We use the software Polymake~\cite{gawrilow2000polymake}, designed for managing polytope and polyhedron.
We implement the Gap function in Python, using the commercial solver Gurobi 11.0~\cite{gurobi} to model and solve the Gap model.

On every vertex, we solve the Gap problem to get the maximum possible value of the integrality gap associated with that vertex. 
\Cref{tab:gaps_dcut_cm} reports our computational results.
From these results, we can draw several conclusions. 
First, Polymake can only exhaustively generate vertices for $n$ up to 5. 
Second, for all these cases, the value of the integrality gap is exactly 1.
For larger values of $n$, the enumeration becomes computationally untractable.
Furthermore, by running the Gap problem on many vertices of the BCR formulation, we observe that the problem turns out to be infeasible.
By analyzing the minimum infeasibility set, we observe that many vertices of the BCR formulation are incompatible with the triangle inequality of the cost vector $c$ nor with its non-negativity.
We tackle both issues in the following sections by (a) introducing a novel formulation tailored for the metric case and (b) designing two heuristic algorithms for enumerating nontrivial vertices.

\begin{table}[!htb]
	\caption{Number of feasible and optimal vertices for $P_{BCR}$ and $P_{CM}$. The column ``time'' reports the time of the generation in seconds, while the column ``gap'' reports the maximum gap obtained for the optimal vertices. While the BCR polytope has several (feasible) vertices that cannot be optimal for any metric cost, the CM polytope does not suffer this issue.}
	\label{tab:gaps_dcut_cm}    
	\centering
	\begin{tabular}{ccrrrrcrrrr}
		\toprule
		&   & \multicolumn{4}{c}{$P_{BCR}$} && \multicolumn{4}{c}{$P_{CM}$} \\
		$n$ &  $t$  & time & feasible & optimal & gap && time & feasible & optimal & gap \\
		\midrule
		4 & 3 &  0.04 & 256 & 70 & 1.00 && 0.73 & 4 & 4 & 1.00\\
		5 & 3 &  4563.57 & 28\,345 & 3\,655 & 1.00 && 44.62 & 5 & 5 & 1.00\\
		5 & 4 &  2798.17 & 24\,297 & 3\,645 & 1.00 && 37.01 & 44 & 44 & 1.00\\
		\bottomrule
	\end{tabular}
	
\end{table}
\subsection{Enumerating vertices with Polymake for the CM formulation}
We enumerate all the vertices of the CM formulation using the software Polymake~\cite{gawrilow2000polymake}. Recalling what we have done in \Cref{sec:gap_DCUT}, we compute the gap value for each vertex by implementing the model \eqref{eq:gap-cm} using Gurobi 11.0.0~\cite{gurobi}.

\Cref{tab:gaps_dcut_cm} reports our computational results.
First, we can observe that the number of vertices generated is smaller than the number of vertices of the BCR formulation, and all of them are feasible. 
As expected, the integrality gap is 1 (Note that it must be a lower bound w.r.t the one of the BCR formulation, which was 1).
Note also that, even in this case, Polymake cannot generate vertices for $n \geq 6$. 
These limited results motivate the design of the two heuristic algorithms to generate nontrivial vertices introduced in the next section. 
\section{Pure one-quarter algorithm}\label{app:POQ}
We present a heuristic for the POQ case that works as follows:
	\begin{enumerate}
		\item Generate all the non-isomorphic graphs having (i) every node of degree at least equal to 3, and (ii) exactly $n + 3t - 4$ arcs.
		\item Filter the list of vertices computed at Step 1 by excluding all graphs with more than $n - t - 1$ nodes with degree 3. In POQ vertices, there are $n - t$ Steiner nodes that must have a minimum degree of 3, and the terminals have a minimum degree of 4.
		\item For each oriented graph, we use {\ttfamily watercluster2} to get all possible orientations of edges, assuming that the maximum indegree must be equal to 4.
		\item We filter out the list obtained at Step 3 and keep only the directed graphs having (i) exactly one node with in-degree 0 (the root), (ii) exactly $t - 1$ nodes with in-degree 4, (iii) exactly $ n - t$ nodes with in-degree 1.
	\end{enumerate}
Hence, we derive \Cref{alg:poq}: a modified version of the PHI algorithm to find POQ vertices exploiting the properties described in \Cref{subsec:beyondphi}.

\begin{algorithm}[!htb]
	\caption{Pure one-quarter vertices search}\label{alg:poq}
	\begin{algorithmic}
		\STATE{\textbf{procedure} POQ$(n,t)$}
		\STATE{$\mathbb{G} = \{G =(V,E) \; | \; G \text{ connected}, \; deg(i) \geq 3 \, \forall i \in V, \; |V| = n, \; |E| = n+3t-4\}$}
		\STATE{$\text{di}\mathbb{G} = \emptyset$}
		\FOR{$G=(V,E) \in \mathbb{G}$}
		\IF{$|\{i \in V \; | \; deg(i) = 3 \}| \leq n -t$}
		\STATE{add to $\text{di}\mathbb{G}$ every non-isomorphic orientation of $G$ s.t.}
		\STATE{\hspace*{4.5em}$\cdot$ every edge can be oriented in only one way}
		\STATE{\hspace*{4.5em}$\cdot$ every node has a maximum indegree of $4$}
		\ENDIF
		\ENDFOR
		\STATE{$\mathcal{V} = \emptyset$}
		\FOR{$\text{di}G = (V, A) \in \text{di}\mathbb{G}$}
		\IF{$|\{i \in V \; | \; indeg(i) = 0 \}| = 1$}
		\IF{$|\{i \in V \; | \; indeg(i) = 1 \}| = n-t$}
		\IF{$|\{i \in V \; | \; indeg(i) = 4 \}| = t-1$}
		\STATE{$x_{ij} = 1/4 \text{ iff } (i,j) \in A$ is a solution of $P_{CM}(n,t)$ with}
		\STATE{\hspace*{7.5em}$\cdot$ $\{r\} = \{i \in V \; | \; indeg(i) = 0 \}$}
		\STATE{\hspace*{7.5em}$\cdot$ $V \setminus T = \{i \in V \; | \; indeg(i) = 1 \}$}
		\STATE{\hspace*{7.5em}$\cdot$ $T \setminus \{r\} = \{i \in V \; | \; indeg(i) = 4 \}$}
		\IF{$x$ is a feasible vertex of $P_{CM}(n,t)$}
		\STATE{add $x$ to $\mathcal{V}$}
		\ENDIF
		\ENDIF
		\ENDIF
		\ENDIF
		\ENDFOR
		\RETURN $\mathcal{V}$
	\end{algorithmic}
\end{algorithm}

\end{document}